\documentclass[11pt,reqno]{amsart}
\usepackage{amsmath, amsfonts, amsthm, amssymb}
\usepackage[usenames]{color}

\usepackage{amssymb}
\usepackage{latexsym}

\newcommand \nc{\newcommand}
\newtheorem{theorem}{Theorem}[section]
\newtheorem{lemma}[theorem]{Lemma}

\newtheorem{corollary}[theorem]{Corollary}
\newtheorem{definition}[theorem]{Definition}

\newtheorem{remark}[theorem]{Remark}

\nc{\ba}{\begin{array}}\nc{\ea}{\end{array}}
\nc{\be}{\begin{eqnarray}}\nc{\ee}{\end{eqnarray}}
\nc{\beq}{\begin{equation}}\nc{\eeq}{\end{equation}}
\nc{\bex}{\begin{eqnarray*}}\nc{\eex}{\end{eqnarray*}}
\nc{\btm}{\begin{theorem}} \nc{\etm}{\end{theorem}}
\nc{\blm}{\begin{lemma}} \nc{\elm}{\end{lemma}}
\nc{\R}{\mathbb{R}}  \nc{\ld}{\lambda}
\nc{\va}{\varphi}
\nc{\ve}{\varepsilon}

\author{Yimei Li, \ Changyou Wang}
\address{School of Mathematical Sciences, Beijing Normal University, Beijing, China 100875}
\email{lym@mail.bnu.edu.cn}
\address{Department of Mathematics, Purdue University, West Lafayette, IN 47907, USA}
\email{wang2482@purdue.edu}

\title[Cosserat Elasticity]
{Regularity of weak solution of variational problems modeling the Cosserat micropolar elasticity}

\begin{document}
\begin{abstract} In this paper, we consider weak solutions of the Euler-Lagrange equation to a variational energy functional modeling the geometrically nonlinear Cosserat micropolar elasticity of continua in dimension three, which
is a system coupling between the Poisson equation and the equation of $p$-harmonic maps ($2\le p\le 3$). We show
that if a weak solutions is stationary, then its singular set is discrete for $2<p<3$ and has zero $1$-dimensional
Hausdorff measure for $p=2$. 
If, in addition, it is a stable-stationary weak solution, then it is regular everywhere when $p\in [2, \frac{32}{15}]$. 
\end{abstract}

\maketitle
\section{Introduction}
\setcounter{equation}{0}
\setcounter{theorem}{0}

General continuum models involving independent rotations were introduced by 
the Eugene and Francois Cosserat brothers in 1909 \cite{Cosserat}, and were later rediscovered in 1960's (see Eringen \cite{Eringen}). The micromorphic balance equations derived by Eringen \cite{Eringen}
were formally justified by \cite{ChenLee1, ChenLee2} as a more realistic continuum model based on molecular dynamics and ensemble averaging. The major difficulty of mathematical treatment in the finite strain case
comes from the geometrically exact formulation of the theory and the appearance of nonlinear manifolds
that are necessary to describe the microstructure.  Among many variants and vast body of results
of Cosserat theory available in the literature, 
P. Neff \cite{Neff1, Neff2, Neff3} has made some systematical analysis of the Cosserat theory for micropolar elastic bodies by establishing the existence of minimizers in the framework of calculus of variations. Very recently,
in an interesting article \cite{Gastel2018}, Gastel has shown a partial regularity theorem of minimizing weak
solutions to a Cosserat energy functional for microplar elastic bodies. 

The elastic body $\Omega\subset\mathbb R^3$ is assumed to be a bounded Lipschitz domain. 
The elastic body can be deformed by a translation mapping $\phi:\Omega\to\mathbb R^3$, and  $\phi(x)-x$ denotes the (small) dislocation for $x\in\Omega$. Furthermore,
the micropolar structure of the material associates each point $x\in\Omega$ with an orthonormal frame 
that is free to rotate in $\mathbb R^3$ by an orthogonal matrix $R(x)\in {SO}(3)$. 
Both translations and rotations induce
material stresses that are given by $R^t \nabla\phi-I_3$ and $R^t \nabla R$ respectively. 
The Cosserat energy functional stored in the elastic body $\Omega$ consists of the contributions
by both translations and rotations.  For a pair of translation and rotation maps $(\phi, R):\Omega\to \mathbb R^3\times SO(3)$, the contribution of rotational stresses to the Cosserat energy is given by
$$\lambda \int_\Omega |R^t\nabla R|^p\,dx\ (=\lambda \int_\Omega |\nabla R|^p\,dx)$$
for some $\lambda>0$ and $2\le p\le 3$, while the contribution of translational stresses is given by
$$\int_\Omega \big|\mathbb P(R^t\nabla \phi-I_3)\big|^2\,dx,$$
where $\mathbb{P}:\mathbb{R}^{3\times 3}\to \mathbb{R}^{3\times 3}$ is the linear map defined by
$$\mathbb{P}(A)=\sqrt{\mu_1} {\rm{devsym}}\ A +\sqrt{\mu_c} {\rm{skew}}\ A +\sqrt{\mu_2} ({\rm{tr}} A) I_3,
\ A\in \mathbb R^{3\times 3},$$
and 
$${\rm{devsym}}\ A=\frac12 (A+A^t)-({\rm{tr}} A) I_3, 
\ \ {\rm{skew}}\ A=\frac12(A-A^t),$$
denotes the deviatoric symmetric part of $A$ and the skew-symmetric part of $A$ respectively. 
The constants $\mu_1, \mu_c,$ and $\mu_2$ are assumed to be positive parameters in this paper.

The elastic body $\Omega$ may be subject to external forces, such as gravity or electromagnetic forces,
that can be modeled by
$$\int_\Omega \langle\phi-x, f\rangle\,dx +\int_\Omega \langle R, M\rangle\,dx,$$
where $f:\Omega\to\mathbb R^3$ and $M:\Omega\to \mathbb R^{3\times 3}$ are given functions.
Collecting together all these terms, the Cosserat energy functional is given by
\begin{equation}\label{Coss}
{\rm{Coss}}(\phi, R)=\int_\Omega \big(|\mathbb P(R^t\nabla\phi-I_3)|^2+|\nabla R|^p+\langle\phi-x, f\rangle 
+\langle R, M\rangle\big)\,dx.
\end{equation}
Recall that $(\phi, R)\in H^1(\Omega,\mathbb R^3)\times W^{1,p}(\Omega, SO(3))$ is a minimizer of
the Cosserat energy functional, if
$$
{\rm{Coss}}(\phi, R)\le {\rm{Coss}}(\widetilde{\phi}, \widetilde{R}),
$$
holds for any $(\widetilde{\phi}, \widetilde{R})\in H^1(\Omega,\mathbb R^3)\times W^{1,p}(\Omega, SO(3))$, with
$(\widetilde{\phi}, \widetilde{R})=(\phi, R)$ on $\partial\Omega$.

The existence of minimizers of ${\rm{Coss}}(\phi, R)$ in the Sobolev spaces, under the Dirichlet boundary condition, has been obtained by Neff \cite{Neff2}.  By direct calculations, any minimizer
$(\phi, R)$ of ${\rm{Coss}}(\phi, R)$ solves the Euler-Lagrange equation, called as the Cosserat equation: 
\begin{equation}\label{Coss1}
\begin{cases} {\rm{div}}\big(R\mathbb P^t\mathbb P(R^t\nabla\phi-I_3)\big)=\frac12 f,\\
\Big\{{\rm{div}}(|\nabla R|^{p-2}\nabla R)-\frac2p\nabla\phi\big(\mathbb P^t\mathbb P(R^t\nabla\phi-I_3)\big)^t-\frac1p M\Big\}\perp T_{R}SO(3).
\end{cases}
\end{equation}
Here $T_{R}SO(3)$ denotes the tangent space of $SO(3)$, at $R\in SO(3)$, that is given by
$$T_RSO(3)=\Big\{ Q\in \mathbb{R}^{3\times 3}\ \big|\  R^t Q+Q^t R=0\Big\},$$
and $\mathbb{P}^t:
\mathbb{R}^{3\times 3}\to \mathbb{R}^{3\times 3}$ is the adjoint map of $\mathbb{P}$.

When $\mu_1=\mu_2=\mu_c=1$, we have that $\mathbb P=\mathbb P^t={\rm{Id}}$ is the identity map.
Hence $$|\mathbb{P}(R^t\nabla \phi-I_3)|^2=|\nabla\phi|^2-2\langle R, \nabla\phi\rangle+3,$$ 
and the Cosserat equation \eqref{Coss1} reduces to the following simplified form:
\begin{equation}\label{Coss2}
\begin{cases} \Delta \phi={\rm{div}} R+\frac12f,\\
\Big({\rm{div}}(|\nabla R|^{p-2}\nabla R)+\frac2p\nabla\phi-\frac{1}{p} M\Big)\perp T_{R}SO(3).
\end{cases}
\end{equation}

We would like to remark that  the system \eqref{Coss1} and \eqref{Coss2} are systems coupling between the Poisson equation
for the macroscopic translational deformation variable $\phi:\Omega\to\mathbb R^3$ and the (nonlinear) $p$-harmonic map equation for the microscopic rotational deformation variable $R:\Omega\to SO(3)$.

By extending the techniques in the study of minimizing $p$-harmonic maps by Schoen-Uhlenbeck \cite{SU1983},
Hardt-Lin \cite{HL1986}, Fuchs \cite{Fuchs},
and especially Luckhaus \cite{Luckhaus}, Gastel has recently shown in an interesting article
\cite{Gastel2018} that any minimizer $(\phi, R)\in H^1(\Omega, \mathbb R^3)\times W^{1,p}(\Omega, SO(3))$ 
of the Cosserat energy functional ${\rm{Coss}}(\phi, R)$ of the Cosserat functional \eqref{Coss}  belongs to $C^{1,\alpha}\times C^\alpha$ in $\Omega$ away from a singular set $\Sigma$ of isolated points for all $2\le p<3$. Moreover, $\Sigma$ is shown to be an empty set
when $p\in [2, \frac{32}{15}]$ by extending
stability inequality arguments by Schoen-Uhlenbeck \cite{SU1984}, Xin-Yang \cite{XY1995},
and Chang-Chen-Wei \cite{CCW2016}.  

An interesting question to ask is whether the regularity result on minimizers of the Cosserat functional 
in \cite{Gastel2018} remains to hold for certain classes of weak solutions to the Cosserat equation
\eqref{Coss1}. In this
paper, we will answer this question affirmatively. To address it, we first need to introduce a few definitions.

For $1\le p<\infty$, recall the Sobolev space
$$W^{1,p}\big(\Omega, SO(3)\big)=\Big\{R\in W^{1,p}(\Omega, \mathbb{R}^{3\times 3})\ \big|
\ R(x)\in SO(3), \ {\rm{a.e.}}\ x\in\Omega\Big\}.
$$
\begin{definition} For $2\le p\le 3$, given $f\in H^{-1}(\Omega, \mathbb R^3)$ and
$M\in W^{-1, \frac{p}{p-1}}(\Omega, \mathbb{R}^{3\times 3})$,
a pair of maps $(\phi, R)\in H^1(\Omega,\mathbb R^3)\times W^{1,p}(\Omega, SO(3))$
is a weak solution to the Cosserat equation \eqref{Coss1}, if it satisfies \eqref{Coss1}
in the sense of distributions, i.e.,
\begin{equation*}
\begin{cases}
\displaystyle\int_\Omega (\langle \mathbb{P}(R^t\nabla\phi-I_3),\mathbb{P}R^t\nabla\psi_1\rangle+\frac12\langle f, \psi_1\rangle)\,dx=0,\\
\displaystyle\int_\Omega\big(\langle |\nabla R|^{p-2}\nabla R,\nabla\psi_2\rangle
+\frac{2}{p}\langle \mathbb{P}(R^t\nabla\phi-I_3), \mathbb{P}\psi_2^t\nabla\phi\rangle+\frac{1}{p}
\langle M, \psi_2\rangle\big)\,dx=0,
\end{cases}
\end{equation*}
hold for any $\psi_1\in H^1_{0}(\Omega,\mathbb{R}^3)$ and
$\psi_2\in W^{1,p}_0(\Omega, T_RSO(3))\cap L^\infty(\Omega, \mathbb{R}^{3\times 3})$. 

\end{definition} 
It is readily seen that any minimizer $(\phi, R)$ of the Cosserat energy functional \eqref{Coss}
is a weak solution of the Cosserat equation \eqref{Coss1}. A restricted class of weak solutions
of \eqref{Coss1} is the class of stationary weak solutions, which is  defined as follows.

\begin{definition} For $2\le p\le 3$, $f\in H^{-1}(\Omega, \mathbb R^3)$, and
$M\in W^{-1, \frac{p}{p-1}}(\Omega, \mathbb{R}^{3\times 3})$,
a weak solution $(\phi, R)\in H^1(\Omega,\mathbb R^3)\times W^{1,p}(\Omega, SO(3))$
to the Cosserat equation \eqref{Coss1} is called a stationary weak solution, if, in addition, $(\phi, R)$ is a critical 
point of the Cosserat energy functional \eqref{Coss} with respect to the domain variations, i.e,
\begin{equation}\label{Coss-Stat}
\frac{d}{dt}\Big|_{t=0} {\rm{Coss}}(\phi_t, R_t)=0,
\end{equation}
where $(\phi_t(x), R_t(x))=(\phi(x+t Y(x)), R(x+t Y(x)))$ for $x\in\Omega$, 
and $Y\in C_0^\infty(\Omega,\mathbb{R}^3)$.
\end{definition}
It is easy to check that any minimizer $(\phi, R)$ of the Cosserat energy functional \eqref{Coss}
is a stationary weak solution of the Cosserat equation \eqref{Coss1}. It can also be shown by a Pohozaev argument
that any regular solution $(\phi, R)\in C^{1,\alpha}(\Omega, \mathbb R^3\times SO(3))$ of the Cosserat equation \eqref{Coss1} is a stationary weak solution.

In section 2 below, we will show that when $\mu_1=\mu_c=\mu_2=1$, 
any stationary weak solution $(\phi, R)$ of Cosserat equation 
\eqref{Coss2} satisfies the following stationarity identity: for any $Y\in C_0^\infty(\Omega,\mathbb R^3)$, 
it holds that
\begin{eqnarray}\label{Coss-Stat1}
&&\int_\Omega \big(|\nabla\phi|^2-2\langle R,\nabla\phi\rangle +|\nabla R|^p\big) (-{\rm{div}} Y)\,dx
+\int_\Omega (\langle f, Y\cdot\nabla\phi\rangle+\langle M, Y\cdot\nabla R\rangle)\,dx\nonumber\\
&&+\int_\Omega\big(2\nabla\phi\otimes\nabla \phi: \nabla Y-2 R_{ij}\frac{\partial\phi^i}{\partial x_k} \frac{\partial Y^k}{\partial x_j}+p|\nabla R|^{p-2}\nabla R\otimes\nabla R:\nabla Y\big)\,dx=0.
\end{eqnarray}
As a direct consequence of \eqref{Coss-Stat1}, we will establish an almost energy monotonicity
inequality for stationary weak solutions to \eqref{Coss2} when $\mu_1=\mu_c=\mu_2=1$ holds. This, combined
with the symmetry of $SO(3)$, enables us to extend the compensated regularity technique by H\'elein 
\cite{Helein}, Evans \cite{Evans}, and Toro-Wang \cite{TW} to show the  following partial regularity.

\begin{theorem} \label{reg1} For $2\le p<3$, $f\in L^\infty(\Omega, \mathbb R^3)$ 
and $M\in L^\infty(\Omega, \mathbb{R}^{3\times 3})$,
if $(\phi, R)\in H^1(\Omega,\mathbb R^3)\times W^{1,p}(\Omega, SO(3))$
is a stationary weak solution to the Cosserat equation \eqref{Coss2},
then there exist $\alpha\in (0,1)$ and a closed set $\Sigma\subset\Omega$,
whose $(3-p)$-dimensional Hausdorff measure $H^{3-p}(\Sigma)=0$, such that
$(\phi, R)\in C^{1,\alpha}(\Omega\setminus\Sigma,\mathbb R^3)\times C^\alpha(\Omega\setminus\Sigma,
SO(3))$.
Furthermore, $\Sigma$ is a discrete set when $p\in (2,3)$.
\end{theorem}

We would like to point out that the discreteness of singular set $\Sigma$ for $2<p<3$ is a corollary
of $H^1\times W^{1,p}$-compactness property of weakly convergent stationary weak solutions of the Cosserat
equation \eqref{Coss2}, which is a consequence of monotonicity inequality \eqref{mono_ineq2}
and the Marstrand Theorem (see \cite {Marstrand}).

To further improve the estimate of the singular set $\Sigma$ for a stationary weak solution $(\phi, R)$
of the Cosserat equation \eqref{Coss1} both for $p=2$ and $2<p<3$, we restrict our attention
to a subclass of stationary weak solutions that are stable. 

\begin{definition} For $2\le p<3$, $f\in H^{-1}(\Omega, \mathbb R^3)$, and
$M\in W^{-1, \frac{p}{p-1}}(\Omega, \mathbb{R}^{3\times 3})$,
a weak solution $(\phi, R)\in H^1(\Omega,\mathbb R^3)\times W^{1,p}(\Omega, SO(3))$
to the Cosserat equation \eqref{Coss1} is called a stable weak solution, if, in addition,
the second order variation of the Cosserat energy functional at $(\phi, R)$ is nonnegative, i.e.,
\begin{equation}\label{Coss-Stable}
\frac{d^2}{dt^2}\Big|_{t=0} {\rm{Coss}}(\phi_t, R_t)\ge 0,
\end{equation}
where $(\phi_t, R_t)\in C^2\big((-\delta, \delta), H^1(\Omega,\mathbb R^3)\times W^{1,p}(\Omega, SO(3))\big)$
for some $\delta>0$,
satisfying $(\phi_0, R_0)=(\phi, R)$,  is a variation of $(\phi, R)$ in the target space $\mathbb R^3\times SO(3)$.
\end{definition}

From the definition, any minimizer $(\phi, R)\in H^1(\Omega,\mathbb R^3)\times W^{1,p}(\Omega, SO(3))$
of the Cosserat energy functional ${\rm{Coss}}(\cdot, \cdot)$ is a stable weak solution of the Cosserat
equation \eqref{Coss1}. 

In section 3, we will establish in the stability Lemma 3.2  that any stable weak solutions $(\phi, R)$ of
Cosserat equation satisfies the following stability inequality: 
\begin{equation}\label{stable0}
\int_\Omega \big((p+1)|\nabla R|^{p-2}|\nabla \psi|^2-2|\nabla R|^p|\psi|^2
\big)\,dx\ge 0, \ \forall\psi\in C^\infty_0(\Omega).
\end{equation}

Utilizing the stability inequality \eqref{stable0}, we can extend the ideas by Hong-Wang \cite{HW1999} and Lin-Wang \cite{LW2006} to establish a pre-compactness property of stable-stationary weak solutions of the Cosserat equation for $p=2$, which can be employed to improve the estimate of singular set $\Sigma$. Moreover, by applying the non-existence theorem
on stable $p$-harmonic maps from $\mathbb S^2$ to $SO(3)$ for $p\in [2,\frac{32}{15}]$ that was established
by Schoen-Uhlenbeck \cite{SU1984}, Xin-Yang \cite{XY1995},
and Chang-Chen-Wei \cite{CCW2016}, we prove a complete regularity result for stable stationary weak solutions
to the Cosserat equation \eqref{Coss2} when $p$ belongs to the range $[2, \frac{32}{15}]$. More precisely,
we have
\begin{theorem}\label{reg2} For $p\in [2,\frac{32}{15}]$, 
$f\in L^\infty(\overline\Omega, \mathbb R^3)$, and
$M\in L^\infty(\overline\Omega, \mathbb{R}^{3\times 3})$,
if $(\phi, R)\in H^1(\Omega,\mathbb R^3)\times W^{1,p}(\Omega, SO(3))$
is a stable stationary weak solution to the Cosserat equation \eqref{Coss2},
then there exists $\alpha\in (0,1)$  such that
$(\phi, R)\in C^{1,\alpha}(\Omega,\mathbb R^3)\times C^\alpha(\Omega, SO(3))$.
\end{theorem}

Now we would like to mention a couple of questions.

\begin{remark} {\rm{1)}}. It remains to be an open question whether Theorem \ref{reg2} remains to be true
when $\frac{32}{15}<p<3$. The main difficulty arises from that we can't rule out the existence of
nontrivial stable $p$-harmonic maps from $\mathbb S^2$ to $SO(3)$ when  $p$ lies in the interval
$(\frac{32}{15}, 3)$.  \\
{\rm{2)}}. It remains to be open whether Theorem \ref{reg1} and Theorem \ref{reg2} hold true 
when the positive constants $\mu_1, \mu_c, \mu_2$ are not necessarily equal. The main difficulty
is that it is unknown whether an almost energy monotonicity inequality holds for stationary weak
solutions $(\phi, R)$ of the Cosserat equation \eqref{Coss1} when $\mathbb P$ is not an identity 
map. 
\end{remark}

The paper is organized as follows. In section 2, we will derive both stationarity identity and an almost energy
monotonicity inequality  for stationary weak solutions $(\phi, R)$ of the Cosserat equation \eqref{Coss1}.
In section 3, we will rewrite the Cosserat equation \eqref{Coss2} into a form in which the nonlinearity exhibits div-curl
structures. In section 4, we will prove an $\epsilon_0$-regularity theorem for stationary weak solutions $(\phi, R)$ of the Cosserat equation \eqref{Coss2}, and apply Marstrand's theorem to obtain a 
refined estimate of the singular set when $2<p<3$. In section 5, we will derive the stability inequality for 
stable weak solutions and obtain the full regularity for stable stationary weak solutions $(\phi, R)$ of the Cosserat equation \eqref{Coss2} when $p\in [2, \frac{32}{15}]$.

\section{Stationarity identity and almost monotonicity inequality}
\setcounter{equation}{0}
\setcounter{theorem}{0}

This section is devoted to the derivation of stationarity identity and almost energy monotonicity
inequality for stationary weak solutions to the Cosserat equation \eqref{Coss2}.

\begin{lemma} \label{stat-id1} For $2\le p<3$, assume $\mu_1=\mu_c=\mu_2=1$,
$f\in L^2(\Omega,\mathbb R^3)$, and $M\in L^{\frac{p}{p-1}}(\Omega,SO(3))$.
If $(\phi, R)\in H^1(\Omega,\mathbb R^3)\times W^{1,p}(\Omega,SO(3))$ is a stationary weak solution of the Cosserat equation \eqref{Coss2}, then for any $Y\in C_0^\infty(\Omega,\mathbb R^3)$ it holds that
\begin{eqnarray}\label{stat-id2}
&&\int_{\Omega}\Big( 2\nabla\phi\otimes\nabla\phi: \nabla Y+p|\nabla R|^{p-2}\nabla R\otimes\nabla R:\nabla Y
-2R^{ik} \frac{\partial\phi^k}{\partial x_j} \frac{\partial Y^j}{\partial x_i}\Big)\,dx\\
&&=\int_{\Omega} \big[\big(|\nabla\phi|^2-2\langle R, \nabla\phi\rangle+|\nabla R|^p\big){\rm{div}}Y-
\langle Y(x)\cdot\nabla\phi,  f\rangle-\langle Y(x)\cdot\nabla R,  M\rangle\big]dx.\nonumber
\end{eqnarray}
\end{lemma}
\begin{proof} For $Y\in C_0^\infty(\Omega,\mathbb R^3)$, there is a sufficiently 
small $\delta>0$ so that ${\rm{dist}}({\rm{supp}}(Y), \partial\Omega)>\delta$. Define
$(\phi_t, R_t)(x)=(\phi,R)(x+tY(x))$ for $x\in\Omega$ and $t\in (-\delta,\delta)$.
Since $(\phi, R)$ is a stationary weak solution of \eqref{Coss2}, we have that 
$$0=\frac{d}{dt}\big|_{t=0} \int_\Omega \big(|\nabla\phi_t|^2-2\langle R_t, \nabla\phi_t\rangle+|\nabla R_t|^p
+\langle \phi_t-x, f\rangle +\langle R_t, M\rangle\big)\,dx.
$$
Applying change of variables and direct calculations, it is not hard to see that 
\begin{eqnarray}
&&\int_\Omega \big(|\nabla\phi|^2-2\langle R,\nabla\phi\rangle +|\nabla R|^p\big) (-{\rm{div}} Y)\,dx
+\int_\Omega \big(\langle Y\cdot\nabla\phi, f\rangle+\langle Y\cdot\nabla R, M\rangle\big)\,dx\nonumber\\
&&+\int_\Omega\big(2\nabla\phi\otimes\nabla \phi: \nabla Y-2 R_{\alpha\beta}\frac{\partial\phi^\beta}{\partial x_\gamma} \frac{\partial Y^\gamma}{\partial x_\alpha}+p|\nabla R|^{p-2}\nabla R\otimes\nabla R:\nabla Y\big)\,dx=0.
\end{eqnarray}
This yields \eqref{stat-id2}.  
\end{proof}

\medskip
By choosing suitable test variation fields $Y\in C_0^\infty(\Omega,\mathbb R^3)$, we will obtain an
almost energy monotonicity inequality for stationary weak solutions to the Cosserat equation \eqref{Coss2}. 

\begin{corollary}\label{mono-ineq1} 
For $2\le p<3$, assume $\mu_1=\mu_c=\mu_2=1$, $f\in L^\infty(\Omega,\mathbb R^3)$
and $M\in L^\infty(\Omega,SO(3))$.
If $(\phi, R)\in H^1(\Omega,\mathbb R^3)\times W^{1,p}(\Omega,SO(3))$ is a stationary weak solution of the Cosserat equation \eqref{Coss2}, then for any $x\in \Omega$ and $0<r_1\le r_2<{\rm{dist}}(x,\partial\Omega)$,
it holds that
\begin{eqnarray}\label{mono_ineq2}
&&{\rm{Coss}}_x((\phi, R), r_1)+\int_{r_1}^{r_2} r^{p-3}\int_{\partial B_r} \big(p |\nabla R|^{p-2}|\frac{\partial R}{\partial r}|^2+ |\frac{\partial\phi}{\partial r}|^2\big)\,dH^2dr\nonumber\\
&&\le {\rm{Coss}}_x((\phi, R), r_2),
\end{eqnarray}
where ${\rm{Coss}}_x((\phi, R), r)$ is the modified renormalized Cosserat energy defined by
\begin{equation}\label{Coss2.0}
{\rm{Coss}}_x((\phi,R),r):= e^{Cr^2} {r}^{p-3}\int_{B_{r}(x)}\big(|\nabla R|^p+ |\nabla\phi|^2\big)\,dx+Cr^{3},
\end{equation}
where $C>0$ depends on $p$, $\|f\|_{L^\infty(\Omega)}$, and $\|M\|_{L^\infty(\Omega)}$.
\end{corollary}

\begin{proof} For simplicity, assume $x=0\in\Omega$ and $0<r<{\rm{dist}}(0,\partial\Omega)$ and write $B_r=B_r(0)$.
Let $Y(x)=x\eta_\epsilon(|x|)$, where $\eta_\epsilon\in C_0^\infty(B_r)$ is chosen such that 
$\eta_\epsilon\rightarrow \chi_{B_r}$ as $\epsilon\rightarrow 0$. Plugging $Y$ into \eqref{stat-id2} and sending
$\epsilon$ to $0$, we obtain that
\begin{eqnarray}\label{mono-ineq20}
&&(p-3)\int_{B_r}|\nabla R|^p\,dx +r\int_{\partial B_r} |\nabla R|^p\,dH^2-\int_{B_r}|\nabla\phi|^2\,dx
+r\int_{\partial B_r} |\nabla\phi|^2\,dH^2\nonumber\\
&&=-4\int_{B_r} \langle R, \nabla\phi\rangle\,dx +2r\int_{\partial B_r} \langle R, \nabla\phi\rangle \,dH^2
-\int_{B_r}(\langle x\cdot\nabla\phi, f\rangle+\langle x\cdot\nabla R, M\rangle)\,dx\nonumber\\
&&\ \ +pr\int_{\partial B_r}|\nabla R|^{p-2}|\frac{\partial R}{\partial r}|^2 \,dH^2
+2r\int_{\partial B_r}|\frac{\partial \phi}{\partial r}|^2\,dH^2
-2\int_{\partial B_r} x^i R^{ik} \frac{\partial\phi^k}{\partial r}\,dH^2. \label{stat-id3}
\end{eqnarray}
It is easy to estimate
\begin{eqnarray*}
&&\big|2r\int_{\partial B_r} \langle R, \nabla\phi\rangle\,dx \big|
\le Cr^{2}\int_{\partial B_r}|\nabla\phi|^2\,dH^2+Cr^2,\\
&&\big|-4\int_{B_r} \langle R, \nabla\phi\rangle\,dx\big|\le Cr\int_{B_r}|\nabla\phi|^2\,dx+Cr^2,\\
&&\big|-2\int_{\partial B_r} R:{x}\otimes \frac{\partial\phi}{\partial r}\,dH^2\big|
\le r\int_{\partial B_r}|\frac{\partial\phi}{\partial r}|^2\,dH^2+Cr^{3},\\
&&\big|-\int_{B_r}\langle x\cdot\nabla\phi, f\rangle\,dx\big|
\le Cr\int_{B_r}|\nabla\phi|^2\,dx+C\|f\|_{L^\infty(\Omega)}^2 r^{4},\\
&&\big|-\int_{B_r}\langle x\cdot\nabla R, M\rangle\,dx\big|
\le Cr\int_{B_r}|\nabla R|^p\,dx+C\|M\|_{L^{\infty}(\Omega)}^{\frac{p}{p-1}} r^{4}.
\end{eqnarray*}
Substituting these estimates into \eqref{mono-ineq20} yields
\begin{eqnarray}\label{mono-ineq21}
&&(p-3)\int_{B_r}|\nabla R|^p\,dx +r\int_{\partial B_r} |\nabla R|^p\,dH^2-\int_{B_r}|\nabla\phi|^2\,dx
+r\int_{\partial B_r} |\nabla\phi|^2\,dH^2\nonumber\\
&&\ge pr\int_{\partial B_r}|\nabla R|^{p-2}|\frac{\partial R}{\partial r}|^2 \,dH^2+r\int_{\partial B_r}|\frac{\partial \phi}{\partial r}|^2\,dH^2\nonumber\\
&&\ \ -Cr\int_{B_r} |\nabla R|^p\,dx -Cr\int_{B_r} |\nabla\phi|^2\,dx-Cr^2\int_{\partial B_r}|\nabla \phi|^2\,dH^2
\nonumber\\
&&\ \ -C\big(1+\|f\|_{L^\infty(\Omega)}^2+\|M\|_{L^\infty(\Omega)}^{\frac{p}{p-1}}\big) r^2.
\end{eqnarray}
Hence we obtain  for $0<r\le \min\big\{1, {\rm{dist}}(0,\partial\Omega)\big\}$,
\begin{eqnarray}\label{stat-id4}
&&\frac{d}{dr}\Big\{e^{Cr^2}r^{p-3}\int_{B_r}\big(|\nabla R|^p+|\nabla\phi|^2\big)\,dx\Big\}\nonumber\\
&&\ge e^{Cr^2} r^{p-3}\int_{\partial B_r} \big(p |\nabla R|^{p-2}|\frac{\partial R}{\partial r}|^2+ 2|\frac{\partial\phi}{\partial r}|^2\big)\,dH^2
+(p-2)e^{Cr^2} r^{p-4}\int_{B_r}|\nabla\phi|^2\,dx\nonumber\\
&&\ \ -C\big(1+\|f\|_{L^\infty(\Omega)}^2+\|M\|_{L^\infty(\Omega)}^{\frac{p}{p-1}}\big) e^{Cr^2} r^2\nonumber\\
&&\ge
r^{p-3}\int_{\partial B_r} \big(p |\nabla R|^{p-2}|\frac{\partial R}{\partial r}|^2+ 2|\frac{\partial\phi}{\partial r}|^2\big)\,dH^2
\nonumber\\
&&\ \ -C\big(1+\|f\|_{L^\infty(\Omega)}^2+\|M\|_{L^\infty(\Omega)}^{\frac{p}{p-1}}\big) r^2.
\end{eqnarray}
Integrating from $0<r_1\le r_2\le \min\{1, {\rm{dist}}(0, \partial\Omega)\}$, 
we obtain that the following monotonicity inequality:
\begin{eqnarray}\label{mono0}
&&e^{Cr_2^2} {r_2}^{p-3}\int_{B_{r_2}}\big(|\nabla R|^p+ |\nabla\phi|^2\big)\,dx+Cr_2^{3}\nonumber\\
&&\ge e^{Cr_1^2} {r_1}^{p-3}\int_{B_{r_1}}\big(|\nabla R|^p+ |\nabla\phi|^2)\,dx
+Cr_1^{3}\nonumber\\
&&\ \ \ +\int_{r_1}^{r_2} r^{p-3}\int_{\partial B_r} \big(p |\nabla R|^{p-2}|\frac{\partial R}{\partial r}|^2+ |\frac{\partial\phi}{\partial r}|^2\big)\,dH^2dr,
\end{eqnarray}
where $C>0$ depends on $p$, $\|f\|_{L^\infty(\Omega)}$, and $\|M\|_{L^\infty(\Omega)}$.
This completes the proof of \eqref{mono_ineq2}. 
\end{proof}

\section{Div-curl structure of the Cosserat equation \eqref{Coss2}}

\setcounter{equation}{0}
\setcounter{theorem}{0}
This section is devoted to rewriting of the Cosserat equation \eqref{Coss2}$_2$ into a form where
the nonlinearity exhibits algebraic structures similar to that of $p$-harmonic maps into
symmetric manifolds given by H\'elein \cite{Helein} and Toro-Wang \cite{TW}.

Let $so(3)$ be the Lie algebra of $SO(3)$ or equivalently the tangent space of
$SO(3)$ at $I_3$. Recall that a standard orthonormal base of $so(3)$ is given by
$$
{\bf a}_1=\frac{1}{\sqrt{2}}\left(\begin{array}{ccc} 0 & 0 & 0 \\
0 & 0 & -1\\
0 & 1 & 0
\end{array}\right), \ 
{\bf a}_2=\frac{1}{\sqrt{2}}\left(\begin{array}{ccc} 0 & 0 & 1 \\
0 & 0 & 0\\
-1 & 0 & 0
\end{array}\right),\
{\bf a}_3=\frac{1}{\sqrt{2}}\left(\begin{array}{ccc} 0 & -1 & 0 \\
1 & 0 & 0\\
0 & 0 & 0
\end{array}\right).
$$
For any $R\in SO(3)$, 
$$\big\{{\bf V}_1(R)={\bf a}_1 R,\ {\bf V}_2(R)={\bf a}_2 R, \ {\bf V}_3(R)={\bf a}_3 R\big\}$$
forms an orthonormal base of $T_R SO(3)$,
the tangent space of $SO(3)$ at $R$.

From \eqref{Coss2}$_2$ we have that for $i=1,2,3$, 
\begin{equation}\label{Coss3.1}
\langle {\rm{div}}(|\nabla R|^{p-2}\nabla R), {\bf V}_i(R)\rangle
=-\frac{2}{p}\langle \nabla\phi,{\bf V}_i(R)\rangle+\frac{1}{p} \langle M, {\bf V}_i(R)\rangle.
\end{equation}
For $i=1,2,3$, since ${\bf a}_i$ is skew-symmetric, we have that
$$
\langle |\nabla R|^{p-2}\nabla R, \nabla({\bf V}_i(R))\rangle=\langle |\nabla R|^{p-2}\nabla R, {\bf a}_i\nabla R\rangle=0.
$$
Thus we can rewrite the Cosserat equation \eqref{Coss2}$_2$ as follows.
\begin{eqnarray}\label{Coss3.2}
&&{\rm{div}}(|\nabla R|^{p-2}\nabla R)=\sum_{i=1}^3 {\rm{div}}\big(\langle |\nabla R|^{p-2}\nabla R, {\bf V}_i(R)\rangle {\bf V}_i(R)\big)\nonumber\\
&&=\sum_{i=1}^3 \big[\langle {\rm{div}}(|\nabla R|^{p-2}\nabla R), {\bf V}_i(R)\rangle +\langle |\nabla R|^{p-2}\nabla R, \nabla({\bf V}_i(R))\rangle
\big]{\bf V}_i(R)\nonumber\\
&&\ \ \ + \sum_{i=1}^3 \langle |\nabla R|^{p-2}\nabla R, {\bf V}_i(R)\rangle \nabla({\bf V}_i(R))\\
&&=\sum_{i=1}^3\Big[\big(-\frac{2}{p}\langle \nabla\phi, {\bf V}_i(R)\rangle+\frac{1}{p}\langle  M, {\bf V}_i(R)\rangle\big)
{\bf V}_i(R)
+  \langle |\nabla R|^{p-2}\nabla R, {\bf V}_i(R)\rangle \nabla({\bf V}_i(R))\Big].\nonumber
\end{eqnarray}
From the above derivation, we see that for $i=1,2,3$, 
\begin{equation}\label{Coss3.3}
{\rm{div}}\big(\langle |\nabla R|^{p-2}\nabla R, {\bf V}_i(R)\rangle\big)
=-\frac{2}{p}\langle\nabla\phi,{\bf V}_i(R)\rangle+\frac{1}{p}\langle M, {\bf V}_i(R)\rangle.
\end{equation} 
For $i=1,2,3$, let $Y_i:\Omega\to\mathbb R$ solve the auxiliary equation
\begin{equation}\label{Coss3.4}
\Delta Y_i=\frac{2}{p}\langle\nabla\phi,{\bf V}_i(R)\rangle-\frac{1}{p}\langle M, {\bf V}_i(R)\rangle,
\end{equation}
so that 
\begin{equation}\label{Coss3.5}
{\rm{div}}\big(\langle |\nabla R|^{p-2}\nabla R, {\bf V}_i(R)\rangle+\nabla Y_i\big)=0.
\end{equation}
Putting \eqref{Coss3.2}, \eqref{Coss3.3}, \eqref{Coss3.4}, \eqref{Coss3.5} together, we obtain
an equivalent form of \eqref{Coss2}$_2$:
\begin{eqnarray}\label{Coss3.6}
&&{\rm{div}}(|\nabla R|^{p-2}\nabla R)\nonumber\\
&&=\sum_{i=1}^3 \big(\langle |\nabla R|^{p-2}\nabla R, {\bf V}_i(R)\rangle+\nabla Y_i\big) \nabla({\bf V}_i(R))
\nonumber\\
&&\ \ \ -\sum_{i=1}^3 \nabla Y_i \cdot\nabla ({\bf V}_i(R))+
\sum_{i=1}^3\big(-\frac{2}{p}\langle \nabla\phi,{\bf V}_i(R)\rangle+\frac{1}{p}\langle M,{\bf V}_i(R)\rangle\big)
{\bf V}_i(R).
\end{eqnarray}
It is readily seen that as the leading order term of nonlinearity in the right hand side of the equation
\eqref{Coss3.6},
$\big(\langle |\nabla R|^{p-2}\nabla R, {\bf V}_i(R)\rangle+\nabla Y_i\big) \nabla({\bf V}_i(R))$ is
the inner product of a divergence free vector field $\big(\langle |\nabla R|^{p-2}\nabla R, {\bf V}_i(R)\rangle+\nabla Y_i\big)$ and a curl free vector field $\nabla({\bf V}_i(R))$. 

\section{$\epsilon_0$-regularity of stationary solutions of the Cosserat equation} 
\setcounter{equation}{0}
\setcounter{theorem}{0}

In this section, we will establish an $\epsilon_0$-regularity estimate and a partial regularity of
stationary weak solutions of the Cosserat equation \eqref{Coss2} and give a proof of Theorem
\ref{reg1}. The key ingredient is the following energy decay lemma, under the smallness condition.

\begin{lemma} \label{decay1}
For any $2\le p<3$, $\mu_1=\mu_c=\mu_2=1$, $f\in L^\infty(\Omega,\mathbb R^3)$ 
and $M\in L^\infty(\Omega,SO(3))$,
there exist $\epsilon_0>0$ and $\theta_0\in (0,\frac12)$ depending on $p$,
$\|f\|_{L^\infty(\Omega)}$, and $\|M\|_{L^\infty(\Omega)}$ 
such that if $(\phi,R)$ is a stationary weak solution of the Cosserat equation \eqref{Coss2}, 
and satisfies, for $x\in\Omega$ and $0<r<{\rm{dist}}(x,\partial\Omega)$,
\begin{equation}\label{small}
r^{p-3}\int_{B_r(x)} \big(|\nabla R|^p+|\nabla\phi|^2\big)\,dx\le\epsilon_0^p,
\end{equation}
then
\begin{eqnarray}\label{decay2}
&&(\theta_0 r)^{p-3}\int_{B_{\theta_0 r}(x)} \big(|\nabla R|^p+|\nabla\phi|^2\big)\,dx\nonumber\\
&&\le \frac12 \max\Big\{ r^{p-3}\int_{B_r(x)} \big(|\nabla R|^p+|\nabla\phi|^2\big)\,dx, \ r^{p}\Big\}.
\end{eqnarray}
\end{lemma}
\begin{proof} We argue it by contradiction. Suppose that the conclusion were false. Then for any $L>0$
with $\|f\|_{L^\infty(\Omega)} +\|M\|_{L^\infty(\Omega)}\le L$ and $\theta\in(0,\frac12)$,
there exist $\epsilon_k\rightarrow 0$, 
$x_k\in\Omega$, and $r_k\rightarrow 0$ such that 
\begin{equation}\label{small1}
r_k^{p-3}\int_{B_{r_k}(x_k)} \big(|\nabla R|^p+|\nabla\phi|^2\big)\,dx\le\epsilon_k^p,
\end{equation}
but
\begin{eqnarray}\label{decay3}
&&(\theta r_k)^{p-3}\int_{B_{\theta r_k}(x_k)} \big(|\nabla R|^p+|\nabla\phi|^2\big)\,dx\nonumber\\
&&> \frac12 \max\Big\{ r_k^{p-3}\int_{B_{r_k}(x_k)} \big(|\nabla R|^p+|\nabla\phi|^2\big)\,dx, 
\  r_k^{p}\Big\}.
\end{eqnarray}
Define the rescaling maps
$$
\begin{cases}
R_k(x)=R(x_k+r_k x), \\
\phi_k(x)=r_k^{\frac{p-2}2}\phi(x_k+r_k x), \\
f_k(x)=r_k^{\frac{p+2}2} f(x_k+r_k x),\\
M_k(x)=r_k^p M(x_k+r_k x),
\end{cases}
\forall x\in B_1.
$$
Then $(\phi_k, R_k)$ solves in $B_1$
\begin{equation}\label{Coss4.0}
\begin{cases}
\Delta \phi_k =r_k^{\frac{p}2}{\rm{div}}(R_k)+\frac12f_k, \\
\displaystyle{\rm{div}}(|\nabla R_k|^{p-2}\nabla R_k)
=\sum_{\alpha=1}^3\langle |\nabla R_k|^{p-2}\nabla R_k, {\bf V}_\alpha(R_k)\rangle \nabla({\bf V}_\alpha(R_k))\\
\displaystyle\ -\frac{1}{p} \sum_{\alpha=1}^3\big[2 r_k^{\frac{p}2}
\langle\nabla\phi_k, {\bf V}_\alpha(R_k)\rangle 
-\langle M_k, {\bf V}_\alpha(R_k)\rangle\big] {\bf V}_\alpha(R_k).
\end{cases}
\end{equation}
Moreover, it holds that
\begin{equation}\label{small2}
\int_{B_1}\big(|\nabla R_k|^p +|\nabla\phi_k|^2\big)\,dx
=r_k^{p-3}\int_{B_{r_k}(x_k)} \big(|\nabla R|^p+|\nabla\phi|^2\big)\,dx
=\epsilon_k^p,
\end{equation}
and
\begin{eqnarray}\label{decay4}
\theta^{p-3}\int_{B_{\theta}} \big(|\nabla R_k|^p+|\nabla\phi_k|^2\big)\,dx
> \frac12 \max\Big\{\int_{B_{1}} \big(|\nabla R_k|^p+|\nabla\phi_k|^2\big)\,dx, \  r_k^{p}\Big\}.
\end{eqnarray}
Now we define the blow-up sequence:
$$
\begin{cases}
\displaystyle\widehat{R_k}(x)=\frac{R_k(x)-\overline{R_k}}{\epsilon_k}, \\
\displaystyle\widehat{\phi_k}(x)=\frac{\phi_k(x)-\overline{\phi_k}}{\epsilon_k^{\frac{p}2}}, 
\end{cases}
\forall x\in B_1,
$$
where $\displaystyle\overline{f}=\frac{1}{|B_1|}\int_{B_1} f$ denotes the average of $f$ over $B_1$. Then $(\widehat{\phi_k}, \widehat{R_k})$ solves, in $B_1$,
\begin{equation}\label{blowupeqn}
\begin{cases}
\Delta \widehat{\phi_k}=r_k^{\frac{p}2}\epsilon_k^{1-\frac{p}2}{\rm{div}}(\widehat{R_k})+\frac12\epsilon_k^{-\frac{p}2}f_k, \\
\displaystyle{\rm{div}}(|\nabla \widehat{R_k}|^{p-2}\nabla \widehat{R_k})=\epsilon_k
\sum_{\alpha=1}^3\langle |\nabla \widehat{R_k}|^{p-2}\nabla \widehat{R_k}, {\bf V}_\alpha(R_k)\rangle \nabla({\bf V}_\alpha(\widehat{R_k}))\\
\displaystyle-\frac{1}{p}\sum_{\alpha=1}^3\big[2 r_k^{\frac{p}2}\epsilon_k^{1-\frac{p}2}
\langle\nabla\widehat{\phi_k}, {\bf V}_\alpha(R_k)\rangle 
-\epsilon_k^{1-p}\langle M_k, {\bf V}_\alpha(R_k)\rangle\big] {\bf V}_\alpha(R_k),
\end{cases}
\end{equation}
satisfies
\begin{equation}\label{small3}
\int_{B_1}\widehat{R_k}\,dx=0, \ \int_{B_1} \widehat{\phi_k}\,dx=0, \
\int_{B_1}\big(|\nabla \widehat{R_k}|^p +|\nabla\widehat{\phi_k}|^2\big)\,dx=1,
\end{equation}
and
\begin{eqnarray}\label{decay5}
\theta^{p-3}\int_{B_{\theta}}\big(|\nabla \widehat{R_k}|^p
+|\nabla\widehat{\phi_k}|^2\big)\,dx> \frac12 \max\Big\{1, \ \frac{r_k^{p}}{\epsilon_k^p}\Big\}.
\end{eqnarray}
In particular, we have
\begin{equation}\label{ratio}
\frac{r_k^{p}}{\epsilon_k^p}\le 2\theta^{p-3}\int_{B_{\theta}} \big(|\nabla \widehat{R_k}|^p
+|\nabla\widehat{\phi_k}|^2\big)\,dx
\le 2\theta^{p-3}.
\end{equation}
This implies that
\begin{equation}\label{ratio1}
r_k\le C\epsilon_k.
\end{equation}
We may assume that there exist 
$\phi_\infty\in H^{1}(B_1,\mathbb R^3)$, $R_\infty\in W^{1,p}(B_1,SO(3))$
such that, after passing to a subsequence, 
$$(\widehat{\phi_k}, \widehat{R_k})\rightharpoonup (\phi_\infty, R_\infty)\ {\rm{in}} \ H^{1}(B_1)\times W^{1,p}(B_1), 
\ (\widehat{\phi_k}, \widehat{R_k})\rightarrow (\phi_\infty, R_\infty)\ {\rm{in}}\ L^2(B_1)\times L^{p}(B_1).$$
Then $(\phi_\infty, R_\infty)$ satisfies  
$$
\begin{cases}
\overline{\phi_\infty}=0,\\
\overline{R_\infty}=0,\\
\displaystyle\int_{B_1}\big(|\nabla R_\infty|^p +|\nabla\phi_\infty|^2\big)\,dx\le 1.
\end{cases}
$$
Moreover, it follows from \eqref{ratio} that
$$\big\|\epsilon_k^{-\frac{p}2} f_k\big\|_{L^\infty(B_1)}\le C\epsilon_k^{-\frac{p}2} r_k^{\frac{p+2}2}
\le Cr_k\rightarrow 0,
$$
$$
\big\|\epsilon_k^{1-p} M_k\big\|_{L^\infty(B_1)}\le C\epsilon_k^{1-p} r_k^p\le C\epsilon_k\rightarrow 0,
$$
and
$$\big\|r_k^{\frac{p}2}\epsilon_k^{1-\frac{p}2} {\rm{div}}(\widehat{R_k})\big\|_{L^p(B_1)}+
\|\big\|r_k^{\frac{p}2} \epsilon_k^{2-p} \nabla\widehat{\phi_k} \big\|_{L^2(B_1)}
\le C\epsilon_k^{1-\frac{p}2} r_k^{\frac{p}2}\le C\epsilon_k\rightarrow 0.
$$
Hence, after sending $k\to\infty$ in the equation \eqref{blowupeqn}, we
conclude that  $\phi_\infty$ is a harmonic function and $R_\infty$ is
a $p$-harmonic function, i.e.,
\begin{equation}\label{limit-eqn2}
\begin{cases}
\Delta\phi_\infty= 0\\
{\rm{div}}\big(|\nabla R_\infty|^{p-2}\nabla R_\infty\big)=0,
\end{cases}
\ {\rm{in}}\ B_1.
\end{equation}
Hence we have that for $0<\theta<\frac12$,
\begin{eqnarray}\label{limit-est}
&&\theta^{p-3}\int_{B_{\theta}} (|\nabla R_\infty|^p+|\nabla\phi_\infty|^2)\,dx\nonumber\\
&&\le C\theta^p \big(\big\|\nabla R_\infty\big\|_{L^\infty(B_\frac12)}^p+
\big\|\nabla \phi_\infty\big\|_{L^\infty(B_\frac12)}^2\big)\nonumber\\
&&\le C\theta^p \int_{B_1}(|\nabla R_\infty|^p +|\nabla\phi_\infty|^2)\,dx\le C\theta^p.
\end{eqnarray}

Next we need to show that $(\widehat{\phi_k}, \widehat{R_k})$ converges strongly to
$(\phi_\infty, R_\infty)$ in $H^1(B_\frac12)\times W^{1,p}(B_\frac12)$, which is based
on the duality between the Hardy space and the BMO space. 
Let $\eta: \R^3 \rightarrow \R$ be a smooth cutoff function satisfying
\[
0\leq \eta\leq 1,\quad \eta=1\ \ \mbox{on} \ B_{\frac14}, \quad  \eta=0\ \ \mbox{on} \ \ \R^3 \backslash B_{\frac38}.
\]
Then we have the following lemma, whose proof is based on the energy monotonicity inequality
\eqref{mono_ineq2} and is similar to that by \cite{Evans} and \cite{TW}.
Denote by ${\rm{BMO}}(\mathbb R^3)$ the space of functions of bounded mean oscillations in $\mathbb R^3$. 

\begin{lemma}\label{bmo}
The sequence  $\{\eta \widehat{R_k}\}_{k\geq1}$ is bounded in {\rm{BMO}}$(\R^3)$.
\end{lemma}
\begin{proof} For the convenience of readers, we sketch the proof here.
Fix any point $x_0\in B_{\frac78}$ and $0<r\leq \frac{1}{8}$, define $
y_k=x_k+r_kx_0\in B_{\frac{7r_k}{8}}(x_k).$
By the monotonicity inequality \eqref{mono_ineq2}, we have
\[
\begin{split}
&\frac{1}{(rr_k)^{3-p}}\int_{B_{rr_k}(y_k)}|\nabla R|^p\,dx
\leq e^{C(rr_k)^2}\frac{1}{(rr_k)^{3-p}}\int_{B_{rr_k}(y_k)}|\nabla R|^p\,dx\\
&\leq e^{2C(rr_k)^2}\frac{1}{(rr_k)^{3-p}}\int_{B_{rr_k}(y_k)}(|\nabla R|^p+|\nabla \phi|^2\,dx
+C(rr_k)^{3}\\
&\leq e^{Cr_k^2}\frac{8^{3-p}}{r_k^{3-p}}\int_{B_{\frac{r_k}8}(y_k)}
(|\nabla R|^p+|\nabla \phi|^2)\,dx+C\big(\frac{1}{8}r_k\big)^{3}\\
&\leq e^{C}\frac{8^{3-p}}{r_k^{3-p}}\int_{B_{r_k}(x_k)}(|\nabla R|^p+|\nabla \phi|^2)\,dx
+C\big(\frac{1}{8}r_k\big)^{3}\\
&\leq e^C8^{3-p} \epsilon_k^p+C8^{-p} r_k^{3}\le C\epsilon_k^p,
\end{split}
\]
where we have used \eqref{ratio} in the last step.
This, combined with the Poincar\'e inequality and the H\"older inequality, yields that
\begin{eqnarray}
\label{1}
&&\Big(\frac{1}{r^3}\int_{B_{r}(x_0)}|\widehat{R_k}-(\widehat{R_k})_{x_0,r}|dx\Big)^p\nonumber\\
&&\le
C\frac{1}{r^{3-p}}\int_{B_{r}(x_0)}|\nabla \widehat{R_k}|^p \nonumber\\
&&= C\frac{1}{\epsilon_k^p(rr_k)^{3-p}}\int_{B_{rr_k}(y_k)}|\nabla R|^p\,dx
\leq C,
\end{eqnarray}
holds for all $k\ge 1$ and all  $x_0\in B_{\frac78}$, $0<r\leq \frac{1}{8}$. 

Applying the John-Nirenberg inequality yields that for any $1\leq q<\infty$,
\[
\big\{\widehat{R_k}\big\}_{k\geq 1}\ \  \mbox{is bounded in} \ \ L^q(B_{\frac78}).
\]
Since $\eta$ is smooth, it follows that for any $y\in B_{r}(x_0)$,
\be\label{3}
\big|(\eta \widehat{R_k})_{x_0,r}-\eta (\widehat{R_k})_{x_0,r}\big|
\leq Cr^{-2}\int_{B_{r}(x_0)}|\widehat{R_k}|\,dx.
\ee
Combining  \eqref{1} with \eqref{3}, it follows that for $x_0\in B_{\frac34}$,
\[
\begin{split}
&\frac{1}{r^3}\int_{B_{r}(x_0)}|\eta \widehat{R_k}-(\eta \widehat{R_k})_{x_0,r}|\,dx\\
&\leq  \frac{1}{r^3}\int_{B_{r}(x_0)}|\eta \widehat{R_k}-\eta(\widehat{R_k})_{x_0,r}|\,dx
+\frac{1}{r^3}\int_{B_{r}(x_0)}|\eta(\widehat{R_k})_{x_0,r}-(\eta \widehat{R_k})_{x_0,r}|\,dx\\
&\leq \frac{1}{r^3}\int_{B_{r}(x_0)}|\widehat{R_k}-(\widehat{R_k})_{x_0,r}|\,dx+
Cr^{-2}\int_{B_{r}(x_0)}|\widehat{R_k}|\,dx\\
&\leq C+\frac{C}{r^{2}}\int_{B_{r}(x_0)}|\widehat{R_k}|\,dx\\
&\leq C+\frac{C}{r^{2}}\big(\int_{B_{\frac78}(x_0)}|\widehat{R_k}|^3\,dx\big)^{\frac13}r^2
\leq C.
\end{split}
\]
Since $\eta=0$ on $\R^3\setminus  B_{\frac38}$,  we have
\[
\sup_k\|\eta \widehat{R_k}\|_{L^1(\R^3)}<\infty.
\]
Hence the above inequality remains to hold for $x_0\in \R^3\setminus  B_{\frac34}$ and $r>0$.  
The proof is complete. \end{proof}

\begin{lemma} \label{strong}
$\nabla  \widehat{R_k}$ converge strongly to 
$\nabla  R_{\infty}$ in $L^p(B_{\frac14})$, 
and $\nabla  \widehat{\phi_k}$ converge strongly to 
$\nabla \phi_{\infty}$ in $L^2(B_{\frac14})$.
\end{lemma}
\begin{proof} 
First notice that scalings of the equation \eqref{Coss3.3} imply that for $i=1,2,3$,
\begin{equation}\label{Coss4.1}
\displaystyle{\rm{div}}\big(\langle |\nabla\widehat{R_k}|^{p-2} \nabla\widehat{R_k}, {\bf V}_i(R_k)\rangle)
=-\frac2p r_k^{\frac{p}2}\epsilon_k^{1-\frac{p}2}\langle\nabla\widehat{\phi_k}, {\bf V}_i(R_k)\rangle
+\frac1{p}\epsilon_k^{1-p}\langle M_k, {\bf V}_i(R_k)\rangle.
\end{equation}
As in \eqref{Coss3.4}, let $Y_k^i:B_1\to\mathbb R$ solve
\begin{equation}\label{Coss4.2}
\begin{cases}
\displaystyle\Delta Y_k^i=\frac2p r_k^{\frac{p}2}\epsilon_k^{1-\frac{p}2}\langle\nabla\widehat{\phi_k}, {\bf V}_i(R_k)\rangle
-\frac1{p}\epsilon_k^{1-p}\langle M_k, {\bf V}_i(R_k)\rangle & {\rm{in}}\ B_1,\\
\ \  Y_k^i=0  & {\rm{on}}\ \partial B_1. 
\end{cases}
\end{equation}
It is easy to see that by $W^{2,2}$-theory, $Y_k^i$ satisfies
\begin{eqnarray}\label{H1-bd}
\big\|\nabla Y_k^i\big\|_{L^2(B_1)}+\big\|\nabla^2 Y_k^i\big\|_{L^2(B_1)}
&\le& Cr_k^{\frac{p}2}\epsilon_k^{1-\frac{p}2}\big\|\nabla\widehat{\phi_k}\big\|_{L^2(B_1)}
+C\epsilon_k^{1-p}\big\|M_k\big\|_{L^2(B_1)}\nonumber\\
&\le& C\big(r_k^{\frac{p}2}\epsilon_k^{\frac{2-p}2}+r_k^p \epsilon_k^{1-p}\big)
\le C\epsilon_k,
\end{eqnarray}
where we have used \eqref{ratio} in the last step.

Adding the equations \eqref{Coss4.1} and \eqref{Coss4.2}, we have that
\begin{equation}\label{div-free}
{\rm{div}}\big(\langle |\nabla\widehat{R_k}|^{p-2} \nabla\widehat{R_k}, {\bf V}_i(R_k)\rangle+\nabla Y_k^i)=0
\ \ {\rm{in}}\ \ B_1,
\end{equation}
and the blowup equation \eqref{blowupeqn}$_2$ becomes
\begin{equation}\label{blowupeqn1}
\begin{cases}
\displaystyle{\rm{div}}(|\nabla \widehat{R_k}|^{p-2}\nabla \widehat{R_k})\\
=\epsilon_k\displaystyle
\sum_{i=1}^3\big(\langle |\nabla \widehat{R_k}|^{p-2}\nabla \widehat{R_k}, {\bf V}_i(R_k)\rangle+\nabla Y_k^i\big) \cdot\nabla({\bf V}_i(\widehat{R_k}))\\
\displaystyle\ \ \ \ -\frac{1}{p}\sum_{i=1}^3\big[2r_k^{\frac{p}2}\epsilon_k^{1-\frac{p}2}\langle\nabla\widehat{\phi_k}, {\bf V}_i(R_k)\rangle -\epsilon_k^{1-p}\langle M_k, {\bf V}_i(R_k)\rangle\big] {\bf V}_i(R_k)\\
\ \ \ \ -\epsilon_k\displaystyle
\sum_{i=1}^3\nabla Y_k^i \cdot\nabla({\bf V}_i(\widehat{R_k}))
\end{cases}
\ \ {\rm{in}}\ \ B_1.
\end{equation}
Define
\[
H_k^i:=\big(\langle |\nabla \widehat{R_k}|^{p-2}\nabla  \widehat{R_k},  {\bf V}_{i}(R_k)\rangle
+\nabla Y_k^i\big)\cdot\nabla ({\bf V}_{i}(\widehat{R_k})).
\]
Then it follows from \eqref{div-free} that $H_k^i\in \mathcal{H}^1_{\rm{loc}}(B_1)$, the local Hardy space
(see \cite{Helein} and \cite{Evans} for some basic properties of Hardy spaces).
For any compact $K\subset B_1$ and $i=1,2,3$, we can use $\frac32<\frac{p}{p-1}\le 2$ and
\eqref{H1-bd} to estimate
\[
\begin{split}
\big\|H_k^i\big\|_{\mathcal{H}^1(K)}&\leq C\big\|\langle |\nabla \widehat{R_k}|^{p-2}\nabla  \widehat{R_k},  
{\bf V}_{i}(R_k)\rangle+\nabla Y_k^i\big\|_{L^{\frac{p}{p-1}}(B_1)}
\big\|\nabla ({\bf V}_{i}(\widehat{R_k})\big\|_{L^{p}(B_1)}\\
&\leq C\big[\|\nabla \widehat{R_k}\|^{p-1}_{L^{p}(B_1)}+\|\nabla {Y}_k^i\|_{L^{\frac{p}{p-1}}(B_1)}\big]
\big\|\nabla ({\bf V}_{i}(\widehat{R_k}))\|_{L^{p}(B_1)}\\
&\leq C, \ \forall k\ge 1.
\end{split}
\]
and
\[
\begin{split}
\big\|H_k^i\big\|_{L^1(B_1)}
&\leq C\big[\|\nabla \widehat{R_k}\|_{L^{p}(B_1)}^{p-1}+\|\nabla {Y}_k^i\|_{L^{\frac{p}{p-1}}(B_1)}\big]
\|\nabla ({\bf V}_{i}(\widehat{R_k}))\|_{L^{p}(B_1)}\\
&\le C, \ \forall k\ge 1.
\end{split}
\]
Assume $\displaystyle\int_{\mathbb R^3}\eta\,dx\ne 0$. For $i=1,2,3$, set 
\[
\mu_k^i=\frac{\int_{\mathbb R^3} H_k^i \eta\,dx}{\int_{\R^3} \eta\,dx}, \ \forall k\ge 1.
\]
Then we have that
\begin{equation}\label{hardy}
\sup_{k\ge 1}\big\|\eta(H_k^i-\mu_k^i)\big\|_{\mathcal{H}^1(\R^3)}\leq C
\sup_{k\ge 1}\big(\|H_k^i\|_{\mathcal{H}^1({\rm{supp}}\eta)}+\|H_k^i\|_{L^1(B_1)}\big)\le C,
\end{equation}
and
\begin{equation}\label{mu-bd}
|\mu_k^i|\le C\|H_k^i\|_{L^1(B_1)}\le C.
\end{equation}
Observe that 
\[
\begin{split}
&\mbox{div}\big(|\nabla \widehat{R_k}|^{p-2}\nabla  \widehat{R_k}-|\nabla R_{\infty}|^{p-2}\nabla  R_{\infty}\big)\\
&=\epsilon_k\sum_{i=1}^3 H_k^i-\epsilon_k  \sum_{i=1}^3\nabla {Y_k^i} \cdot \nabla ({\bf V}_{i}(\widehat{R_k}))\\
&\ \ \ -\frac{1}{p}\sum_{i=1}^3\big[2r_k^{\frac{p}2}\epsilon_k^{1-\frac{p}2}\langle\nabla\widehat{\phi_k}, {\bf V}_i(R_k)\rangle -\epsilon_k^{1-p}\langle M_k, {\bf V}_i(R_k)\rangle\big] {\bf V}_i(R_k).
\end{split}
\]
Multiplying this equation by $\eta^2(\widehat{R_k}-R_\infty)$ and integrating it over $\mathbb R^3$,
we obtain that 
\[
\begin{split}
&\int_{B_1}\eta^2(|\nabla \widehat{R_k}|^{p-2}\nabla  \widehat{R_k}-|\nabla R_{\infty}|^{p-2}\nabla  R_{\infty}):\nabla(\widehat{R_k}-R_\infty) \,dx\\
&+2\int_{B_1}\eta (|\nabla \widehat{R_k}|^{p-2}\nabla  \widehat{R_k}-|\nabla R_{\infty}|^{p-2}\nabla  R_{\infty}):
\nabla\eta\otimes(\widehat{R_k}-R_\infty) \,dx\\
&=\epsilon_k\int_{B_1}\big[\nabla Y_k^i \cdot\nabla ({\bf V}_{i}(\widehat{R_k}))-H_k^i\big]\eta^2(\widehat{R_k}-R_\infty)\, dx\\
&+\frac{1}{p}\sum_{i=1}^3\int_{B_1}\big[2r_k^{\frac{p}2}\epsilon_k^{1-\frac{p}2}\langle\nabla\widehat{\phi_k}, {\bf V}_i(R_k)\rangle -\epsilon_k^{1-p}\langle M_k, {\bf V}_i(R_k)\rangle\big] {\bf V}_i(R_k)\eta^2(\widehat{R_k}-R_\infty)\, dx.
\end{split}
\]
It is not hard to see that
\[
\begin{split}
&\int_{B_1}\eta^2|\nabla {\widehat{R_k}}-\nabla R_\infty|^{p} \,dx\\
&\leq C\int_{B_1}\eta|(|\nabla \widehat{R_k}|^{p-2}\nabla \widehat{R_k}
-|\nabla {R}_\infty|^{p-2}\nabla {R}_\infty)||\nabla\eta||\widehat{R_k}-{R}_\infty|\,dx\\
&\ +C\epsilon_k\big|\int_{B_1} H_k^i \cdot \eta^2(\widehat{R_k}-R_\infty)\,dx\big|
+C\epsilon_k\int_{B_1} \eta^2|\nabla {Y}_k^i| |\nabla \widehat{R_k}||\widehat{R_k}-R_\infty|\,dx\\
&\ +C\int_{B_1}\big[2r_k^{\frac{p}2}\epsilon_k^{1-\frac{p}2}|\nabla\widehat{\phi_k}|
+\epsilon_k^{1-p}|M_k|\big] \eta^2|\widehat{R_k}-R_\infty|\, dx\\
&=I_k+II_k+III_k+IV_k.
\end{split}
\]
Since 
$$|\nabla\widehat{R_k}|^{p-2}\nabla\widehat{R_k} 
\rightharpoonup |\nabla R_\infty|^{p-2}\nabla R_\infty\ \ {\rm{in}}\ \ L^{\frac{p}{p-1}}(B_1),
\ \ \widehat{R_k}\rightarrow R_\infty \ \ {\rm{in}}\ \ L^{p}(B_1),
$$
we conclude that
$$I_k\rightarrow 0.$$
For $III_k$, we have
\begin{eqnarray*}
|III_k|&\le & C\epsilon_k \|\nabla Y_k^i\|_{L^6(B_1)} \|\nabla \widehat{R_k}\|_{L^2(B_1)}
\|\widehat{R_k}-R_\infty\|_{L^3(B_1)} \\
&\le & C\epsilon_k \|\nabla Y_k^i\|_{H^1(B_1)} \|\nabla \widehat{R_k}\|_{L^2(B_1)}
\|\widehat{R_k}-R_\infty\|_{L^3(B_1)}\\
&\le& C\epsilon_k\rightarrow 0.
\end{eqnarray*}
We can apply \eqref{ratio} to estimate $IV_k$ by
\begin{eqnarray*}
|IV_k|&\le & Cr_k^{\frac{p}2}\epsilon_k^{1-\frac{p}2}\|\nabla\widehat{\phi_k}\|_{L^2(B_1)}
\|\widehat{R_k}-R_\infty\|_{L^2(B_1)}+Cr_k^p\epsilon_k^{1-p}\|\widehat{R_k}-R_\infty\|_{L^1(B_1)}\\
&\le & Cr_k^{\frac{p}2}\epsilon_k^{\frac{2-p}2}
\|\widehat{R_k}-R_\infty\|_{L^2(B_1)}+Cr_k^p\epsilon_k^{1-p}\|\widehat{R_k}-R_\infty\|_{L^1(B_1)}\\
&\le & C\epsilon_k\|\widehat{R_k}-R_\infty\|_{L^2(B_1)}\rightarrow 0.
\end{eqnarray*}
While the most difficult term $II_k$ can be estimated by employing the duality between $\mathcal{H}^1(\R^3)$
and ${\rm{BMO}}(\R^3)$ as follows.
\[
\begin{split}
&\int_{B_1}H_k^i \eta^2(\widehat{R_k}-R_\infty)\,dx\\
&=\int_{B_1(0)}\eta(H_k^i-\mu_k^i)\eta(\widehat{R_k}-R_\infty)\,dx
+\mu_k^i\int_{B_1}\eta^2(\widehat{R_k}-{R}_\infty)\,dx\\
&=V_k+VI_k.
\end{split}
\]
It is easy to estimate
\[
|VI_k|\le C|\mu_k^i| \int_{B_1}|\widehat{R_k}-R_\infty|\,dx
\leq C\|H_k^i\|_{L^1(B_1)}\int_{B_1}|\widehat{R_k}-R_\infty|\,dx\rightarrow 0.
\]
We can apply Lemma \ref{bmo} and \eqref{hardy} and
\eqref{mu-bd} to estimate $V_k$ by
\[
\begin{split}
&|V_k|=\big|\int_{B_1}\eta(H_k^i-\mu_k^i)\eta(\widehat{R_k}-R_\infty)\,dx\big|\\
&\leq C\big\|\eta(H_k^i-\mu_k^i)\big\|_{\mathcal{H}^1(\R^3)}\big\|\eta(\widehat{R_k}-R_{\infty})
\|_{{\rm{BMO}}(\R^3)}\leq C.
\end{split}
\]
Therefore we obtain that
$$
|II_k|\le C\epsilon_k(|V_k|+|VI_k|)\le C\epsilon_k\rightarrow 0.
$$
Putting all the estimates of $I_k, II_k, III_k, IV_k$ together, we arrive that 
\[
\int_{B_{\frac14}}|\nabla (\widehat{R_k}-R_\infty)|^{p} \,dx\to 0.
\]

Next, we are going to prove that
\[
\nabla  \widehat{\phi_k} \longrightarrow \nabla  \phi_{\infty}\quad\mbox{in }\ \ L^2(B_{\frac14}).
\]
Since 
\[
-\Delta \widehat{\phi_k}=r_k^{\frac{p}{2}}\epsilon_k^{1-\frac{p}2}{\rm{div}}( \widehat{R_k})
+\frac12\epsilon_k^{-\frac{p}2} f_k
\quad\mbox{in} \ \ B_1,
\]
and
\[
-\Delta  \phi_{\infty}=0\quad\mbox{in} \ \ B_1,
\]
multiplying both equations by $\eta^2(\widehat{\phi_k}-\phi_\infty)$, subtracting the resulting equations,
and integrating over $\R^3$, we obtain that 
\[
\begin{split}
&\int_{B_1}\eta^2|\nabla  (\widehat{\phi_k}-\phi_\infty)|^2\,dx
+2\int_{B_1}\eta \nabla  (\widehat{\phi_k}-\phi_\infty)
\nabla \eta(\widehat{\phi_k}-\phi_\infty) \,dx\\
&=\int_{B_1}\big[r_k^{\frac{p}{2}}\epsilon_k^{1-\frac{p}2}{\rm{div}}( \widehat{R_k})
+\frac12\epsilon_k^{-\frac{p}2} f_k\big] 
\eta^2(\widehat{\phi_k}-\phi_\infty)\,dx.
\end{split}
\]
Since $\widehat{\phi_k} \longrightarrow \phi_{\infty}$
and $
\nabla  \widehat{\phi_k} \rightharpoonup \nabla  \phi_{\infty}$ in  $L^2(B_{\frac14})$,
we conclude that
$$2\int_{B_1}\eta \nabla  (\widehat{\phi_k}-\phi_\infty)
\nabla \eta(\widehat{\phi_k}-\phi_\infty) \,dx\rightarrow 0.$$
Also, since
\begin{eqnarray*}
&&\big\|r_k^{\frac{p}{2}}\epsilon_k^{1-\frac{p}2}{\rm{div}}( \widehat{R_k})+\frac12\epsilon_k^{-\frac{p}2} f_k\big\|_{L^2(B_1)}\\
&&\le C r_k^{\frac{p}2}\epsilon_k^{1-\frac{p}2}\|\nabla \widehat{R_k}\|_{L^2(B_1)}
+C\epsilon_k^{-\frac{p}2} r_k^{\frac{p+2}2}\le C\epsilon_k
\rightarrow 0,
\end{eqnarray*}
we conclude that
$$\int_{B_1}\big[r_k^{\frac{p}{2}}\epsilon_k^{1-\frac{p}2}{\rm{div}}( \widehat{R_k})
+\frac12\epsilon_k^{-\frac{p}2} f_k\big] 
\eta^2(\widehat{\phi_k}-\phi_\infty)\,dx\rightarrow 0.$$
Thus we obtain that 
\[
\int_{B_{\frac14}}|\nabla (\widehat{\phi_k}-\phi_\infty)|^{2}\,dx\to 0.
\]
This completes the proof of Lemma \ref{strong}.  \end{proof}

Now we return to the proof of Lemma \ref{decay1}. It follows from Lemma \ref{strong} and the estimate
\eqref{limit-est} that for sufficiently large $k>1$, it holds that 
\[
\begin{split}
\theta^{p-3}\int_{B_{\theta}}\big(|\nabla \widehat{R_k}|^{p} +|\nabla \widehat{\phi_k}|^{2}\big)
\le C\theta^p+o(1) \le \frac{1}{2}\max\big\{1, \frac{r_k^{p}}{\epsilon_k^p}\big\},
\end{split}
\]
provided that $0<\theta<\frac14$ is chosen to be sufficiently small.
This contradicts to the assumed inequality \eqref{decay5}.  Hence the proof of Lemma \ref{decay1} is
complete. 
\end{proof}

Next we apply Lemma \ref{decay1} and the Marstrand Theorem to give a proof of
Theorem \ref{reg1}. 

\medskip
\noindent{\it Proof of Theorem} \ref{reg1}.  Define the singular set $\Sigma$ by
$$\Sigma=\Big\{x\in \Omega\ \big|\  \Theta^{3-p}((\phi, R),x) \equiv\lim_{r\to 0} {\rm{Coss}}_x\big((\phi, R),r\big) \ge \frac12\epsilon_0^p\Big\}.$$
Here ${\rm{Coss}}_x\big((\phi, R), r\big)$ denotes the modified renormalized Cosserat energy
of $(\phi, R)$ in $B_r(x)$ defined by \eqref{Coss2.0}, which is monotonically increasing with resepct
to $r>0$ by Corollary \ref{mono-ineq1}. Hence the density function 
$$\Theta^{3-p}((\phi, R), x)=\lim_{r\to 0} {\rm{Coss}}_x\big((\phi, R),r\big)$$
exists for any $x\in\Omega$ and is upper semicontinuous  in $\Omega$.  From a simple covering argument
(see \cite{Evans-Gariepy}), we know that 
$$H^{3-p}(\Sigma)=0.$$
For any $x_0\in \Omega\setminus\Sigma$, there exists $r_0>0$ such that $B_{r_0}(x_0)\subset\Omega$,
and
$$
{\rm{Coss}}_{x_1}((\phi,R),\frac{r_0}2)= e^{Cr_0^2} {(\frac{r_0}2)}^{p-3}\int_{B_{\frac{r_0}2}(x_1)}\big(|\nabla R|^p+ |\nabla\phi|^2\big)\,dx+C(\frac{r_0}2)^{3}\le \epsilon_0^p
$$
holds for all $x_1\in B_{\frac{r_0}2}(x_0)$.

Applying Lemma \ref{decay1} repeatedly, we would obtain that there exists $\theta_0\in (0,\frac12)$ such that
\begin{eqnarray}\label{decay-l}
&&(\theta_0^l r_0)^{p-3}\int_{B_{\theta_0^l r_0}(x_1)} (|\nabla R|^p+|\nabla\phi|^2)\,dx\nonumber\\
&&\le 2^{-l} \max\Big\{r_0^{p-3}\int_{B_{r_0}(x_0)}(|\nabla R|^p+|\nabla\phi|^2)\,dx,\ \ 
\frac{Cr_0^{p}}{1-2\theta_0^{p}}\Big\}
\end{eqnarray}
holds for all $x_1\in B_{\frac{r_0}2}(x_0)$ and $l\ge 1$. 

It follows from \eqref{decay-l} that there exists $\alpha_0\in (0,1)$ such that 
\begin{eqnarray}\label{morrey-decay}
&&r^{p-3}\int_{B_{r}(x_1)} (|\nabla R|^p+|\nabla\phi|^2)\,dx\nonumber\\
&&\le \big(\frac{r}{r_0}\big)^{p\alpha_0} \max\Big\{r_0^{p-3}\int_{B_{r_0}(x_0)}(|\nabla R|^p+|\nabla\phi|^2)\,dx,\ \ \frac{Cr_0^{p}}{1-2\theta_0^{p}}\Big\}\nonumber\\
&&\le C(\epsilon_0) \big(\frac{r}{r_0}\big)^{p\alpha_0}
\end{eqnarray}
holds for all $x_1\in B_{\frac{r_0}2}(x_0)$ and $0<r\le \frac{r_0}2$. Thus, by Morrey's decay Lemma
\cite{Evans-Gariepy}, we conclude that $(\phi, R)\in C^{\alpha_0}(B_{\frac{r_0}2}(x_0))$. Since
$$\Delta\phi={\rm{div}}(R)+\frac12f \ \ {\rm{in}}\ \ B_{r_0}(x_0),$$
the higher order regularity theory of Poisson equation implies that $\phi\in C^{1,\alpha_0}(B_{\frac{r_0}2}(x_0))$. Since $x_0\in\Omega\setminus\Sigma$ is arbitrary, we obtain that $(\phi, R)\in C^{1,\alpha_0}(\Omega\setminus\Sigma)\times  C^{\alpha_0}(\Omega\setminus\Sigma)$. 

Next we will employ the Marstrand Theorem \cite{Marstrand} to show that the singular set 
$\Sigma$ is discrete for $2<p<3$.  We argue it by contradiction.  Suppose $\Sigma$ is not
discrete.  Then there exist a sequence of points $\{x_k\}\subset \Sigma $ and $x_0\in \Sigma$
such that $x_k\rightarrow x_0$. Set $r_k=|x_k-x_0|\rightarrow0$ and define
\[
 (\phi_k, \ R_k, \ f_k, \ M_k)(x)=(r_k^{\frac{p-2}2}\phi, \ R, \ r_k^{\frac{p+2}2} f,\ r_k^p M)(x_0+r_kx),\ \forall x\in B_2.
\]
It is readily seen that $(\phi_k, R_k)$ is singular at $0$ and $y_k=\frac{x_k-x_0}{r_k}\in \mathbb S^2$.
Moreover, similar to \eqref{Coss4.0}, $(\phi_k, R_k)$ solves 
\begin{equation}\label{Coss4.10}
\begin{cases}
\Delta \phi_k =r_k^{\frac{p}2}{\rm{div}}(R_k)+\frac12f_k, \\
\displaystyle{\rm{div}}(|\nabla R_k|^{p-2}\nabla R_k)
=\sum_{i=1}^3\langle |\nabla R_k|^{p-2}\nabla R_k, {\bf V}_i(R_k)\rangle \nabla({\bf V}_i(R_k))\\
\displaystyle\ -\frac{1}{p} \sum_{i=1}^3\big[2 r_k^{\frac{p}2}
\langle\nabla\phi_k, {\bf V}_i(R_k)\rangle 
-\langle M_k, {\bf V}_i(R_k)\rangle\big] {\bf V}_i(R_k)
\end{cases}
\ \ {\rm{in}}\ \  B_2.
\end{equation}
It follows from the monotonicity inequality \eqref{mono_ineq2} for $(\phi, R)$ and the scaling argument
that $(\phi_k,R_k)$ also enjoys the following monotonicity inequality, i.e., for $0<r_1<r_2\le 2$
\begin{equation}\label{scale-mono}
\begin{split}
&e^{Cr_1^2}r_1^{p-3}\int_{B_{r_1}}(|\nabla R_k|^p+|\nabla \phi_k|^2)\,dx+Cr_1^{3}\\
&+\int_{r_1}^{r_2}r^{p-3}\int_{\partial B_r}\big(p|\nabla R_k|^{p-2}\big|\frac{\partial R_k}{\partial r}\big|^2
+\big|\frac{\partial \phi_k}{\partial r}\big|^2\big)\,dH^2dr\\
&\leq e^{Cr_2^2}r_2^{p-3}\int_{B_{r_2}}\big(|\nabla R_k|^p+|\nabla \phi_k|^2)\,dx+Cr_2^{3}.
\end{split}
\end{equation}
Moreover, for $k>1$ sufficiently large,
\begin{eqnarray}\label{up-low}
\frac14\epsilon_0^p&\le& 2^{p-3}\int_{B_2}(|\nabla R_k|^p+|\nabla \phi_k|^2)\,dx\nonumber\\
&=&(2r_k)^{p-3}\int_{B_{2r_k}(x_0)}(|\nabla R|^p+|\nabla \phi|^2)\,dx\le C.
\end{eqnarray}
Hence
\[
\int_{B_2}(|\nabla R_k|^p+|\nabla \phi_k|^2)\,dx\quad\mbox{is uniformly bounded above and below}.
\]
Then there exists $( \phi_{\infty}, R_{\infty})\in H^{1}(B_2, \R^3)\times W^{1,p}(B_2, SO(3))$ such that,
after passing to a subsequence, 
\[
 (\phi_k, R_k) \rightharpoonup  (\phi_{\infty}, R_\infty)\quad\mbox{in }\ \ H^{1}(B_2)\times W^{1,p}(B_2).
\]
It is not hard to see that by passing the limit $k\to\infty$ in \eqref{Coss4.10}, we see that
$\phi_\infty$ is a harmonic function in $B_2$, and $R_\infty$ is a $p$-harmonic map into $SO(3)$
in $B_2$.  Moreover, it follows from the lower semicontinuity and the monotonicity inequality
\eqref{scale-mono} that for any $0<s\le 2$, it holds
$$
\int_{s}^{2}r^{p-3}\int_{\partial B_r}\big(p|\nabla R_\infty|^{p-2}\big|\frac{\partial R_\infty}{\partial r}\big|^2+
\big|\frac{\partial \phi_\infty}{\partial r}\big|^2\big)\,dH^2dr=0,
$$
this follows from the fact that for any fixed $0<s\le 2$, 
$$
e^{Cs^2}s^{p-3}\int_{B_{s}}(|\nabla R_k|^p+|\nabla \phi_k|^2)\,dx+Cs^{3}\to \Theta^{3-p}\big((\phi, R), x_0\big),
\ {\rm{as}}\ k\to \infty.
$$
Therefore we must have that 
\[
(\frac{\partial \phi_\infty}{\partial r}, \ \frac{\partial R_\infty}{\partial  r})=(0,\ 0),
\]
or equivalently $(\phi_\infty, R_\infty)$ is homogeneous of degree zero:
\begin{equation}\label{deg-zero}
\big(\phi_\infty(x), R_\infty(x)\big)=\big(\phi_\infty(\frac{x}{|x|}), R_\infty(\frac{x}{|x|})\big),  \ x\in B_2.
\end{equation}
Since $\phi_\infty$ is a smooth harmonic function with homogeneous degree zero, it follows
that $\phi_\infty$ is a constant.

Next we need to show\\
{\it Claim 1}. 
\[
 (\phi_k, R_k) \longrightarrow  (\phi_{\infty}, R_\infty)\quad\mbox{in }\ \ H^{1}(B_1)\times W^{1,p}(B_1).
\]

Assume the claim for the moment. Then it follows from \eqref{up-low} and $\phi_\infty=$ constant
that $R_\infty:B_2\to SO(3)$ is a nontrivial stationary $p$-harmonic map, which has at least two
singular points $0$ and $y_\infty\in \mathbb S^2$ given by
\[
y_\infty=\lim_{k\to \infty}\frac{x_k-x_0}{|x_k-x_0|}.
\]
The singular set of $R_\infty$ contains the line segement $[0y_\infty]$
so that $H^{1}({\rm{Sing}}(R_\infty))>0$, which is impossible. Thus $\Sigma$ is a discrete set.

Finally, we would like to apply Marstrand theorem to prove {\it Claim 1}.
To do it,  we consider a sequence of Radon measures 
\[
\mu_k=(|\nabla R_k|^p+|\nabla \phi_k|^2)\,dx.
\]
Since $\mu_k(B_2)$ is uniformly bounded, we may assume that there is a nonnegative Radon measure 
$\mu$ in $B_2$ such that after passing to a subsequence, 
$$\mu_k\rightarrow \mu$$
as convergence of Radon measures. By Fatou's lemma, we can decompose $\mu$ into
\[
\mu=(|\nabla R_\infty|^p+|\nabla \phi_\infty|^2)\,dx+\nu
\]
for a nonnegative Radon measure $\nu$, called a defect measure.
The monotonicity inequality \eqref{scale-mono} for $(\phi_k, R_k)$ 
implies that $\mu$ is a monotone measure in the following sense:
for $x\in B_1$, $0<r_1<r_2<{\rm{dist}}(x,\partial B_2)$,
\[
e^{Cr_1^2}r_1^{p-3}\mu(B_{r_1}(x))+Cr_1^{3}\leq e^{Cr_2^2}r_2^{p-3}\mu(B_{r_2}(x))+Cr_2^{3}.
\]
In particular, for any $x\in B_1$, the density function
\[
\Theta^{3-p}(\mu,x)=\lim_{r\rightarrow 0}r^{p-3}\mu(B_r(x))
\]
exists  and is upper semicontinuous in $B_1$.  Define the concentration set 
\[
\boldsymbol{ S}:=\bigcap_{r>0}\left\{x\in B_1\ \big|\ 
\liminf_{k\rightarrow \infty}r^{p-3}\int_{B_r(x)}(|\nabla R_k|^p+|\nabla \phi_k|^2)\,dx\geq \frac12\epsilon_0^p\right\}.
\]
We claim that $\boldsymbol{S}$ is a closed subset of $B_1$.  In fact, let $\{x_k\}$ be a sequence of
points in  $\boldsymbol{S}$ such that $x_k\to x_0\in B_1$. 
If $x_0\not\in \boldsymbol{ S}$, then there exists $r_0>0$ and $\delta_0>0$
such that for $k>1$ sufficiently large it holds that 
\[
r_0^{p-3}\int_{B_{r_0}(x_0)}(|\nabla R_k|^p+|\nabla \phi_k|^2)\,dx\le
\frac12\epsilon_0^p-\delta_0.
\]
Taking $k$ large enough so that $|x_k-x_0|<\frac{r_0}2$ and applying the monotonicity inequality to each 
$(\phi_{k}, R_{k})$,  we have
\[
\begin{split}
&e^{C(\frac{r_0}{2})^2}\big(\frac{r_0}{2}\big)^{p-3}\int_{B_{\frac{r_0}2}(x_k)}(|\nabla R_{k}|^p+|\nabla \phi_{k}|^2)\,dx+C\left(\frac{r_0}{2}\right)^{3}\\
&\leq e^{C(r_0-|x_k-x_0|)^2}(r_0-|x_k-x_0|)^{p-3}
\int_{B_{r_0-|x_k-x_0|}(x_k)}(|\nabla R_{k}|^p+|\nabla \phi_{k}|^2)\,dx\\
&\ \ +C\left((r_0-|x_k-x_0|\right)^{3}\\
&\leq e^{C(r_0-|x_k-x_0|)^2}\big(\frac{r_0}{(r_0-|x_k-x_0|)}\big)^{3-p}
r_0^{p-3}\int_{B_{r_0}(x_0)}|\nabla R_{k}|^p+|\nabla \phi_{k}|^2)\,dx\\
&\ \ +C\left((r_0-|x_k-x_0|\right)^{3}\\
&\leq e^{Cr_0^2}\big(\frac{r_0}{(r_0-|x_k-x_0|)}\big)^{3-p}(\frac12\epsilon_0^p-\delta_0)
+C\left((r_0-|x_k-x_0|\right)^{3}\\
&\le \frac12\epsilon_0^p,
\end{split}
\]
provided that $k$ large enough and $r_0$ is chosen sufficiently small.
This contradicts to the fact $x_k\in \boldsymbol{S}$. Hence $\boldsymbol{S}$
is a closed subset.

Suppose $x_*\in B_1\setminus \boldsymbol{S}$. Then there exists $r_*>0$ such that 
\[
\liminf_{k\rightarrow \infty}{r_*}^{p-3}\int_{B_{r_*}(x_*)}
(|\nabla R_k|^p+|\nabla \phi_k|^2)\,dx<\frac12\epsilon_0^p.
\]
Applying the $\epsilon_0$-regularity Theorem \ref{reg1}, we may conclude that after passing to
another subsequence, 
\[
R_k\longrightarrow R_{\infty}\quad\mbox{in }\ \ C^1_{\rm{loc}}\cap W_{\rm{loc}}^{1,p}(B_1\setminus \boldsymbol{ S}),
\]
and
\[
 \phi_k \longrightarrow  R_{\infty}\quad\mbox{in }\ \ C^1_{\rm{loc}}\cap H_{\rm{loc}}^{1}(B_1\setminus \boldsymbol{ S}).
\]
If we denote by ${\rm{Sing}}(\phi_\infty,R_\infty)$ the set of discontinuity of $(\phi_\infty,R_\infty)$, and  
${\rm{supp}}(\nu)$ the support of the defect measure $\nu$.  
Then the above convergence implies that 
$${\rm{Sing}}(\phi_\infty,R_\infty)\bigcup{\rm{supp}}(\nu)\subset \boldsymbol{ S}.$$
On the other hand, if $\hat{x}\in  \boldsymbol{S}$, then after sending $k\to \infty$, we have that 
\[
\frac{\mu(B_r(\hat{x}))}{r^{3-p}}\geq \frac12{\epsilon_0^p}, \ \forall r>0.
\]
If $\hat{x} \not\in {\rm{Sing}}( \phi_{\infty}, R_{\infty})$, then 
$( \phi_{\infty}, R_{\infty})$  is regular near $\hat{x}$ and hence for $r$ sufficiently small,
\[
r^{p-3}\int_{B_r(\hat{x})}(|\nabla R_\infty|^p+|\nabla \phi_\infty|^2)dx\leq \frac14{\epsilon_0^p},
\]
this implies that for small $r>0$,
\[
\frac{\nu(B_r(\hat{x}))}{r^{3-p}}\geq \frac14{\epsilon_0^p},
\]
and hence $\hat{x}\in {\rm{supp}}(\nu)$. Therefore, we conclude that
\begin{lemma}\label{cc3}
\[
{\rm{Sing}}(\phi_\infty,R_\infty)\bigcup {\rm{supp}}(\nu)=\boldsymbol{S}.
\]
\end{lemma}
Notice that if $x\in \boldsymbol{S}$, then \[
 \Theta^{3-p}(\mu,x)=\lim_{r\to 0} r^{3-p}\mu(B_r(x))\geq \frac12\epsilon_0^p.
\]
Moreover, for any compact subset $K\subset \subset  B_1$,  and any $x\in \boldsymbol{S}\cap K$, 
\[
\frac12\epsilon^p_0\leq \Theta^{3-p}(\mu,x)\leq r_K^{p-3}\mu(B_2)\le r_K^{p-3} E_0,
\]
where $r_K=\frac{1}{2}{\rm{dist}}(K,\partial B_2)>0$, and
$\displaystyle E_0=\sup_k\int_{B_2}(|\nabla R_k|^p+|\nabla \phi_k|^2)\,dx$.
Recall that by Federer-Ziemer theorem (see \cite{Evans-Gariepy})
\[
\lim_{r\rightarrow 0}r^{p-3}\int_{B_r(x)}(|\nabla R_\infty|^p+|\nabla \phi_\infty|^2)\,dy=0
\]
holds for $H^{3-p}$ a.e. $x\in B_2$. Thus we obtain that
\begin{lemma}\label{conc}
For any compact $K\subset B_1$, if $x\in \boldsymbol{S}\cap K$, then 
\[
\frac12\epsilon_0^p\leq \Theta^{3-p}(\mu,x)<C(K)<\infty.
\]
For $H^{3-p}$ a.e. $x\in \boldsymbol{S}$,
\[
\Theta^{3-p}(\mu,x)=\Theta^{3-p}(\nu,x).
\]
\end{lemma}
It follows from Lemma \ref{conc} and standard covering arguments that  for any compact set
$K\subset B_1$
\[
\epsilon^pH^{3-p}(\boldsymbol{S}\cap K)\leq 
\nu(\boldsymbol{S}\cap K)\leq CH^{3-p}(\boldsymbol{ S}\cap K).
\]
Therefore,
\[
\nu(\boldsymbol{ S})=0 \iff H^{3-p}(\boldsymbol{S})=0.
\]
In particular, we have that 
\begin{lemma}\label{c5}
$(\phi_k, R_k)\nrightarrow (\phi_\infty, R_\infty)$ strongly in $H^1(B_1)\times W^{1,p}(B_1)$ if
and only if $\nu(B_1)>0$ if and only if $H^{3-p}(\boldsymbol{S})>0$.
\end{lemma}

Return to the proof of {\it Claim 1}. For $2<p<3$, if
\[
(\phi_k, R_k) \nrightarrow (\phi_\infty, R_{\infty})\ {\rm{in}}\ H^1(B_1)\times W_{loc}^{1,p}(B_1),
\]
then  by Lemma \ref{c5},  we must  have $H^{3-p}(\boldsymbol{S})>0$. Hence by Lemma \ref{cc3}
we have for $H^{3-p}$ a.e. $x\in \boldsymbol{S}$,
\[
0<\Theta^{3-p}(\nu,x)<\infty.
\]
Applying Marstrand Theorem to $\nu$ and $\boldsymbol{S}$, we
conclude that $3-p$ must be an integer, which is impossible.  Hence {\it Claim 1} is true.
This completes the proof of Theorem \ref{reg1}. \qed

\section{Stable-stationary solutions of the Cosserat equation}
\setcounter{equation}{0}
\setcounter{theorem}{0}

This section is devoted to the proof of Theorem \ref{reg2}. More precisely, we will show that if
$(\phi, R)$ is a stable stationary solution to the Cosserat equation \eqref{Coss2}. Then the singular
set is empty for $p$ belonging to the range $[2, \frac{32}{15}]$.

It is well-known that $\mathbb{S}^3$ is the universal cover of $SO(3)$, and a locally isometric 
$2$-to-$1$ covering map $\pi:\mathbb{S}^3\to SO(3)$ is given by
$$
\pi(w,x,y,z)=\left(\begin{array}{ccc} 1-2y^2-2z^2 & 2xy-2zw & 2xz+2yw\\
2xy+2zw & 1-2x^2-2z^2 & 2yz-2xw\\
2xz-2yw & 2yz+2xw & 1-2x^2-2y^2\end{array}
\right), \ \forall (w,x,y,z)\in\mathbb{S}^3.
$$
In particular, the curvature operator of $SO(3)$, $R_{SO(3)}$, satisfies
$$
\langle R_{SO(3)}(v,w)v, w\rangle =|v|^2|w|^2-\langle v, w\rangle^2, \ v, w\in T_{R}SO(3).
$$
 
For  $(\phi, R)\in H^1(\Omega,\mathbb{R}^3)\times W^{1,p}(\Omega, SO(3))$,
let 
$$(\phi_t, R_t)\in C^2((-\delta,\delta), H^1(\Omega,\mathbb R^3)\times W^{1,p}(\Omega,SO(3)))$$
be a family of variations of $(\phi, R)$. 
Denote by
$$\eta=\frac{d}{dt}|_{t=0}\phi_t, \ \hat{\eta}=\frac{d^2}{dt^2}|_{t=0}\phi_t,$$
and
$$v=\frac{\partial R_t}{\partial t}\big|_{t=0}, \ \hat{v}=\nabla_{\frac{\partial}{\partial t}}
\frac{\partial R_t}{\partial t}\big|_{t=0}.$$
Applying the equation \eqref{Coss2} for $(\phi, R)$ and direct calculations as in Smith \cite{Smith},
we obtain that
\begin{eqnarray*}
&&\frac{d^2}{dt^2}\big|_{t=0} {\rm{Coss}}(\phi_t, R_t)\\
&&=\frac{d^2}{dt^2}\big|_{t=0} \int_\Omega \big(|\nabla\phi_t|^2
-2\langle R_t, \nabla\phi_t\rangle +|\nabla R_t|^p+(\phi_t-x)\cdot f+\langle R_t, M\rangle\big)\,dx\\
&&=\int_\Omega \big(2|\nabla\eta|^2
-4\langle v,\nabla\eta\rangle +p|\nabla R|^{p-2}(|\nabla v|^2-{\rm{tr}}\langle R_{SO(3)}(v, \nabla R)v,\nabla R\rangle)\\
&&\ \ \qquad+p(p-2)|\nabla R|^{p-4}\langle \nabla R, \nabla v\rangle^2\big)\,dx\\
&&=\int_\Omega \big(2|\nabla\eta|^2
-4\langle v,\nabla\eta\rangle +p|\nabla R|^{p-2}(|\nabla v|^2-|\nabla R|^2|v|^2)\\
&&\ \ \qquad+p(p-2)|\nabla R|^{p-4}\langle \nabla R, \nabla v\rangle^2 \big)\,dx
\end{eqnarray*}
holds for any $\eta\in H^1_0(\Omega,\R^3)$ and $v\in H^1_0\cap L^\infty(\Omega, T_{R}SO(3))$.

\begin{definition} For $2\le p<3$, $\mu_1=\mu_c=\mu_2=1$,
$f\in L^\infty(\Omega,\R^3)$ and $M\in L^\infty(\Omega, SO(3))$, 
a stationary weak solution $(\phi, R)$ of the Cosserat equation
\eqref{Coss2} is called a stable, stationary weak solution of the Cosserat equation
\eqref{Coss2}
if, in addition, 
$$
\frac{d^2}{dt^2}\big|_{t=0} {\rm{Coss}}(\phi_t, R_t)\ge 0,
$$
or, equivalently,
\begin{eqnarray}\label{stable1}
&&\int_\Omega \big(2|\nabla\eta|^2
-4\langle v,\nabla\eta\rangle +p|\nabla R|^{p-2}(|\nabla v|^2-|\nabla R|^2|v|^2)\nonumber\\
&&\qquad+p(p-2)|\nabla R|^{p-4}\langle \nabla R, \nabla v\rangle^2\big)\,dx\ge 0
\end{eqnarray}
holds for any $\eta\in C_0^\infty(\Omega,\mathbb R^3)$ and
$v\in H^1_0(\Omega, T_{R} SO(3))$.
\end{definition}

\begin{lemma} For $2\le p<3$, $\mu_1=\mu_c=\mu_2=1$,
$f\in L^\infty(\Omega,\R^3)$ and $M\in L^\infty(\Omega, SO(3))$, if $(\phi, R)$ is a stable, stationary weak solution of the Cosserat equation
\eqref{Coss2}, then 
\begin{equation}\label{stable2}
\int_\Omega \big(6|\nabla\omega|^2
-4\sum_{i=1}^3 \psi\langle {\bf a}_i R,\nabla\omega\otimes {\bf e}^i\rangle 
+p(p+1)|\nabla R|^{p-2}|\nabla \psi|^2-2p|\nabla R|^p|\psi|^2
\big)\,dx\ge 0
\end{equation}
holds for any $\omega\in C_0^\infty(\Omega)$ and
$\psi\in C^\infty_0(\Omega)$. Here $({\bf e}^1, {\bf e}^2, {\bf e}^3)$ is the standard base of
$\mathbb R^3$. In particular, 
\begin{equation}\label{stable3}
\int_\Omega \big((p+1)|\nabla R|^{p-2}|\nabla \psi|^2-2|\nabla R|^p|\psi|^2
\big)\,dx\ge 0
\end{equation}
holds for any 
$\psi\in C^\infty_0(\Omega)$.
\end{lemma}
\begin{proof} It is readily seen that \eqref{stable3} follows immediately from \eqref{stable2} by taking $\omega=0$.
Thus it suffices to show \eqref{stable2}. For any $\omega\in C_0^\infty(\Omega)$ and
$\psi\in C^\infty_0(\Omega)$, 
let $\eta=\omega {\bf e}^i$ and $v=\psi {\bf a}_i R$ and substitute them into \eqref{stable1} and then
take summation over $i=1,2,3$ , we obtain that
\begin{eqnarray}\label{stable4}
&&\int_\Omega \big(2\sum_{i=1}^3|\nabla(\omega {\bf e}^i)|^2
-4\sum_{i=1}^3 \psi\langle {\bf a}_i R,\nabla\omega\otimes {\bf e}^i\rangle+p(p-2)|\nabla R|^{p-4}
\sum_{i=1}^3\langle\nabla R, \nabla(\psi{\bf a}_i R)\rangle^2 \nonumber\\
&&\ \ +p|\nabla R|^{p-2}\sum_{i=1}^3(|\nabla (\psi{\bf a}_i R)|^2-|\nabla R|^2|\psi{\bf a}_i R|^2)
\big)\,dx\ge 0.
\end{eqnarray}
Observe that
\begin{eqnarray*}
&&\sum_{i=1}^3\langle\nabla R, \nabla(\psi{\bf a}_i R)\rangle^2
=\sum_{i=1}^3\big[\sum_{j=1}^3\nabla_j\psi\langle\nabla_j R, {\bf a}_i R\rangle+
\psi\langle\nabla R, {\bf a}_i \nabla R\rangle\big]^2\\
&&=\sum_{i=1}^3 \langle\nabla\psi\cdot\nabla R, {\bf a}_i R\rangle^2
=|\nabla\psi\cdot\nabla R|^2\le |\nabla\psi|^2|\nabla R|^2,\\
&&\sum_{i=1}^3|\nabla(\omega {\bf e}^i)|^2=3|\nabla\omega|^2,
\ \ \sum_{i=1}^3|\nabla R|^2|\psi{\bf a}_i R|^2=3|\nabla R|^2|\psi|^2,
\end{eqnarray*}
and
\begin{eqnarray*}
&&\sum_{i=1}^3|\nabla (\psi{\bf a}_i R)|^2\\
&&=|\nabla\psi|^2 \sum_{i=1}^3 |{\bf a}_i R|^2
+2\psi\nabla\psi \sum_{i=1}^3\langle {\bf a}_i R, {\bf a}_i \nabla R\rangle 
+|\psi|^2\sum_{i=1}^3\langle {\bf a}_i\nabla R, {\bf a}_i \nabla R\rangle\\
&&=3|\nabla\psi|^2+\psi\nabla\psi {\rm{tr}}(R^T {\bf a}_i^T{\bf a}_i \nabla R
+\nabla R^T {\bf a}_i^T {\bf a}_i R) +|\psi|^2 
{\rm{tr}}\big(\nabla R^T\nabla R (\sum_{i=1}^3 {\bf a}_i^T {\bf a_i})\big)\\
&&= 3|\nabla\psi|^2+\psi\nabla\psi {\rm{tr}}[(R^T  \nabla R
+\nabla R^T R)({\bf a}_i^T {\bf a}_i)] +|\psi|^2 
{\rm{tr}}\big(\nabla R^T\nabla R (\sum_{i=1}^3 {\bf a}_i^T {\bf a_i})\big)\\
&&=3|\nabla\psi|^2+|\nabla R|^2|\psi|^2,
\end{eqnarray*}
where we have used 
$${\bf a}_1^T {\bf a}_1={\rm{diag}}(0,\frac12,\frac12), 
\ {\bf a}_2^T {\bf a}_2={\rm{diag}}(\frac12,0,\frac12),\ {\bf a}_3^T {\bf a}_3={\rm{diag}}(\frac12,\frac12,0),$$
and
$$\langle R, {\bf a}_i \nabla R\rangle=0, \ R^T\nabla R+\nabla R^T R=0.$$
Plugging these identities into \eqref{stable4}, we obtain \eqref{stable2}.
\end{proof}

Now we can extend the partial regularity theorem for stationary weak solutions of the Cosserat euqation
\eqref{Coss2} obtained in the previous section
to the class of stable weak solutions of the Cosserat euqation
\eqref{Coss2}. First, we consider Theorem \ref{reg2} in the case that $p=2$.
Namely, we will show that

\begin{theorem} For $f\in L^\infty(\Omega,\mathbb R^3)$ and $M\in L^\infty(\Omega, SO(3))$,
and $\mu_1=\mu_c=\mu_2=1$, assume that $(\phi, R)\in H^1(\Omega,\mathbb R^3)\times H^1(\Omega, SO(3))$
is a stable, stationary  weak solution of the Cosserat euqation
\eqref{Coss2}. Then $(\phi, R)\in C^{1,\alpha}(\Omega, \mathbb R^3)\times C^\alpha(\Omega, SO(3))$
for some $\alpha\in (0,1)$.
\end{theorem}
\begin{proof} From the small energy regularity theorem obtained in the previous section, we know that
there exists a closed singular set $\Sigma\subset\Omega$, with $H^1(\Sigma)=0$, such that
$(\phi, R)\in C^{1,\alpha}(\Omega\setminus\Sigma)\times C^\alpha (\Omega\setminus\Sigma)$ for some $0<\alpha<1$.

Now we want to show $\Sigma=\emptyset$. For, otherwise, there exists $x_0\in\Sigma$ such that
$$\Theta^1((\phi, R), x_0)\equiv\lim_{r\downarrow 0} 
r^{-1}\int_{B_r(x_0)} (|\nabla R|^2+|\nabla\phi|^2)\,dx
\ge \epsilon_0^2>0.$$
For any sequence of radius $r_i\downarrow 0$, define the blow up sequence
$$(\phi_i, R_i, f_i,M_i)(x)=(\phi,R, r_i^2 f, r_i^2M)(x_0+r_i x), \ \forall x\in B_2.$$
Then
$$\lim_{i\to\infty} 2^{-1}\int_{B_2} (|\nabla R_i|^2+|\nabla\phi_i|^2)\,dx
=\Theta^1((\phi, R), x_0)\ge \epsilon_0^2.
$$
Thus there exists $(\phi_0, R_0)\in H^1(B_2,\mathbb R^3)\times H^1(B_2, SO(3))$ 
such that after passing to a  subsequence, 
$$(\phi_i, R_i)\rightharpoonup (\phi_0, R_0) \ {\rm{in}}\ H^1(B_2,\mathbb R^3)\times H^1(B_2, SO(3)).$$
Since $(\phi_i, R_i)$ satisfies
\begin{equation}\label{cosserat}
\begin{cases}\Delta\phi_i=r_i {\rm{div}} R_i+\frac12 f_i\\
\Delta R_i+\frac{2}{p} r_i\nabla\phi_i-\frac{1}{p} M_i \perp T_{R_i} SO(3),
\end{cases}
\end{equation}
it follows, after sending $i\to\infty$, that on $B_2$, $\phi_0$ is a harmonic function  and
$R_0$ is a harmonic map into $SO(3)$. We now need

\medskip
\noindent{\it Claim 2}: $(\phi_i, R_i)\rightarrow (\phi_0, R_0) \ {\rm{in}}\ H^1(B_1,\mathbb R^3)\times H^1(B_1, SO(3)).$

\medskip
We will apply the technique of potential theory 
by Hong-Wang \cite{HW1999} and Lin-Wang \cite{LW2006} to prove this claim.
Let $\nu\ge 0$ be a Radon measure in $B_2$ such that
$$\mu_i\equiv(|\nabla R_i|^2+|\nabla\phi_i|^2)\,dx
\rightharpoonup \mu\equiv (|\nabla R_0|^2+|\nabla\phi_0|^2)\,dx+\nu
$$
as convergence of measures in $B_2$. It suffices to show
$\nu\equiv 0$ in $B_1$.  Notice that $(\phi_i,R_i)$, solving \eqref{cosserat}, is indeed a stationary 
weak solution of the Euler-Lagrange equation of critical point of the Cosserart energy functional
$$
E_i(\widehat{\phi}, \widehat{R})=\int_{B_2} (|\nabla \widehat{R}|^2
+|\nabla\widehat{\phi}|^2-2r_i\langle \widehat{R}, \nabla \widehat{\phi}\rangle
+(\widehat{\phi}-x)\cdot f_i+\langle\widehat{R}, M_i\rangle)\,dx.
$$
In particular, the $\epsilon_0$-regularity theorem is applicable to $(\phi_i, R_i)$ and we conclude that if
we define
\begin{eqnarray*}
\mathcal{S}&=&\bigcap_{r>0}
\Big\{y\in B_{\frac32}: \  \lim_{i\to\infty} r^{-1}\int_{B_r(y)} (|\nabla R_i|^2+|\nabla\phi_i|^2)\,dx\ge \epsilon_0^2\Big\}\\
&=&\Big\{y\in B_{\frac32}: \  \Theta^1(\mu, y)=\lim_{r\to 0}
r^{-1}\mu({B_r(y)})\ge \epsilon_0^2\Big\}.
\end{eqnarray*}
Then the following statements hold:\\
(i) $\mathcal{S}$ is closed with $H^1(\mathcal{S})<\infty$,  ${\rm{supp}}(\nu)\subset \mathcal{S}$
and $\Theta^1(\nu,y)=\Theta^1(\mu, y)\ge \epsilon_0^2$ for $H^1$ a.e. $y\in\mathcal{S}$.  \\
(ii) There exists $\alpha\in (0,1)$ such that 
$$(\phi_i, R_i)\rightarrow (\phi_0, R_0) \ {\rm{in}}\ 
(C^\alpha_{\rm{loc}}\cap H^1_{\rm{loc}})(B_\frac32\setminus\mathcal{S}).
$$
(iii) 
$$C_1(\epsilon_0) H^1(\mathcal{S})\le \nu(B_\frac32)\le C_2(\epsilon_0) H^1(\mathcal{S}).$$
In particular, $\nu\equiv 0$ if  and only if $H^1(\mathcal{S})=0$.  It follows
from $H^1(\mathcal{S})<+\infty$ that $Cap_2(\mathcal{S})=0$. Hence for
any $\delta>0$, there exists $\omega_\delta\in C^\infty_0(B_2)$ such that
$$\mathcal{S}\subset {\rm{int}}(\{\omega_\delta=1\}),$$
and
\begin{equation}\label{cap2}
\int_{B_2} |\nabla\omega_\delta|^2\,dx\le \delta.
\end{equation}
Hence for any $a\in\mathcal{S}$, there exists $0<r_a<\delta^2$  such that
$$\omega_\delta\ge \frac12 \ {\rm{on}}\ B_{r_a}(a).$$ 
From the compactness of $\mathcal{S}$ and Vitali's covering lemma, 
there exist $1\le l<\infty$ and $\{a_m\}_{m=1}^l\subset\mathcal{S}$ such that
$\{B_{\frac{r_{a_m}}5}(a_m)\}_{m=1}^l$ are mutually disjoint, and
$$\mathcal{S}\subset \bigcup_{m=1}^l B_{r_{a_m}}(a_m).$$
From the definition of $\mathcal{S}$, there exists a sufficiently large 
$i_l>0$ such that
\begin{equation}\label{lower-bd}
\frac{\epsilon_0^2} 2\le \big(\frac{r_{a_m}}{5}\big)^{-1} \int_{B_{\frac{r_{a_m}}{5}}(a_m)} 
(|\nabla R_i|^2+|\nabla \phi_i|^2)\,dx,
 \ \forall i\ge i_l, \ m=1,\cdots, l.
\end{equation}
By the $W^{1,q}$-estimate on $\phi$, we know that 
$$\|\nabla \phi\|_{L^q(K)}\le C(q, K)$$
holds for any compact set $K\Subset \Omega$ and $1<q<\infty$.
Hence for any $i\ge i_l$ and $m=1,\cdots, l$, it follows from H\"older's inequality that
\begin{eqnarray*}
\big(\frac{r_{a_m}}{5}\big)^{-1} \int_{B_{\frac{r_{a_m}}{5}}(a_m)} |\nabla \phi_i|^2\,dx
&\le& 
C\big({r_ir_{a_m}}\big)^{-1}\int_{B_{r_i r_{a_m}}(x_0+r_ia_m)} |\nabla\phi|^2\,dx\\
&\le& C(q) (r_i r_{a_m})^{2-\frac{6}{q}}
\le C\delta^{\frac32}\le \frac14{\epsilon_0^2},
\end{eqnarray*}
provided we choose $q=12$ and $\delta\le \displaystyle\big(\frac{\epsilon_0^2}{4C}\big)^\frac23$ in the last step.
Substituting this estimate into \eqref{lower-bd}, we obtain that 
\begin{equation}\label{lower-bd1}
\frac14{\epsilon_0^2} \le \big(\frac{r_{a_m}}{5}\big)^{-1} \int_{B_{\frac{r_{a_m}}{5}}(a_m)} |\nabla R_i|^2\,dx,
 \ \forall i\ge i_l, \ m=1,\cdots, l.
\end{equation}
Therefore for all $i\ge i_l$, we can bound
\begin{eqnarray}\label{size}
H^1_{\delta^2}(\mathcal{S})&\le & C\sum_{m=1}^l r_{a_m}= 5C\sum_{m=1}^l \frac{r_{a_m}}5\nonumber\\
&\le & \frac{20 C}{\epsilon_0^2}\sum_{m=1}^l\int_{B_{\frac{r_{a_m}}{5}}(a_m)} |\nabla R_i|^2\,dx\nonumber\\
&\le &\frac{80 C}{\epsilon_0^2}\int_{\bigcup_{m=1}^l B_{\frac{r_{a_m}}{5}}(a_m)} |\nabla R_i|^2\omega_\delta^2\,dx
\nonumber\\
&\le & \frac{80 C}{\epsilon_0^2}\int_{B_2} |\nabla R_i|^2\omega_\delta^2\,dx.
\end{eqnarray}
It follows from the stability of $(\phi, R)$ and a scaling argument  that $R_i$ satisfies the
stability inequality \eqref{stable3} so that
\begin{equation}\label{stable5}
\int_{B_2} |\nabla R_i|^2\omega_\delta^2\,dx
\le \frac32\int_{B_2}|\nabla\omega_\delta|^2\,dx, \ \forall i\ge i_l.
\end{equation}
Plugging \eqref{stable5} into \eqref{size} and applying \eqref{cap2}, we would obtain that
$$
H^1_{\delta^2}(\mathcal{S})\le C(\epsilon_0) \delta.
$$
This, after sending $\delta\to 0$, would yield $H^1(\mathcal{S})=0$ and hence {\it Claim 2} is true.

It follows from the $H^1$-strong convergence of $(\phi_i,R_i)$ to $(\phi_0, R_0)$ and the energy monotonicity
inequality \eqref{mono0}, we conclude that 
$$(\phi_0,R_0)(x)=(\phi_0, R_0)(\frac{x}{|x|}), \ \forall x\in B_2, $$
is homogeneous of degree zero. Since $\phi_0$ is a harmonic function in $B_2$, it follows
that $\phi_0$ is a constant. Thus
$$\int_{\mathbb S^2}|\nabla_{\mathbb S^2} R_0|^2\,dH^2=\Theta^1((\phi, R), x_0)\ge \epsilon_0^2,$$
and $R_0\in C^\infty(\mathbb S^2, SO(3))$ is a nontrivial harmonic map. Since
$\Pi_1(\mathbb S^3)=\{0\}$, it follows that
there exists a nontrivial harmonic map $\widehat{R_0}\in C^\infty(\mathbb S^2, \mathbb S^3)$ 
such that $R_0=\pi\circ \widehat{R_0}$. Moreover, it follows from the stability inequality \eqref{stable1}
that $\widehat{R_0}$ is a stable harmonic map from $\mathbb S^2$ to $\mathbb S^3$, i.e.
\begin{equation}\label{stable6}
\int_{\mathbb S^2}\big(|\nabla_{\mathbb S^2} \omega|^2-|\nabla\widehat{R_0}|^2|\omega|^2\big)\,dH^2\ge 0
\end{equation}
for any $\omega\in C^\infty(\mathbb S^2, T_{\widehat{R_0}}\mathbb S^3)$. However it follows
from Schoen-Uhlenbeck \cite{SU1984} that there is no nontrivial stable harmonic map from
$\mathbb S^2$ to $\mathbb S^3$. We get a desired contradiction. Thus
the singular set $\Sigma$ of $(\phi, R)$ is empty. \end{proof}

\medskip
Theorem \ref{reg2} for the cases that $p>2$ can be summarized into the following theorem.

\begin{theorem} For $f\in L^\infty(\Omega,\mathbb R^3)$ and $M\in L^\infty(\Omega, SO(3))$,
and $\mu_1=\mu_c=\mu_2=1$, if $p\in (2, \frac{32}{15}]$ and
 $(\phi, R)\in H^1(\Omega,\mathbb R^3)\times W^{1,p}(\Omega, SO(3))$
is a stable, stationary weak solution of the Cosserat equation \eqref{Coss2}, 
then there exists $\alpha\in (0,1)$ such that 
$(\phi, R)\in C^{1,\alpha}(\Omega, \mathbb R^3)\times C^{\alpha}(\Omega, SO(3))$.
\end{theorem}

\begin{proof} It follows from $2<p<3$ and Theorem \ref{reg1} that ${\rm{Sing}}(\phi, R)$ is discrete.
Suppose ${\rm{Sing}}(\phi, R)\not=\emptyset$. Then there exist $x_0\in {\rm{Sing}}(\phi, R)$ and
$r_0>0$ such that ${\rm{Sing}}(\phi, R)\cap B_{r_0}(x_0)=\{x_0\}$. For $r_k\rightarrow 0$, define
$(\phi_k, R_k)(x)=(\phi, R)(x_0+r_k x)$ for $x\in B_2$. As in the proof of Theorem \ref{reg1}, we can
apply the monotonicity inequality \eqref{mono_ineq2}, Lemma \ref{decay1}, and Marstrand theorem to show that 
there exists a nontrivial $(\phi_0, R_0)\in H^1(B_1, \R^3)\times W^{1,p}(B_1, SO(3))$ 
such that, after passing to a subsequence, $(\phi_k, R_2)\rightarrow (\phi_0, R_0)$ strongly in $H^1(B_1, \R^3)\times W^{1,p}(B_1, SO(3))$. Hence $(\phi_0, R_0)$ is of homogeneous degree zero, $\phi_0$ is constant
and $R_0\in C^{1,\alpha}(B_1\setminus\{0\}, SO(3))$ is a nontrivial, stable, stationary $p$-harmonic map. 
However, it follows from the stability Lemma 6.3 and Proposition 6.4 in Gastel \cite{Gastel2018}
that for $p\in (2, \frac{32}{15})$, any stable stationary $p$-harmonic map
$R(x)=R(\frac{x}{|x|})\in C^{1,\alpha}(B_1\setminus\{0\}, SO(3))$ must be constant. We get a desired contradiction.
Hence ${\rm{Sing}}(\phi, R)=\emptyset$ when $p\in (2, \frac{32}{15}]$. This completes the proof.
\end{proof} 

Finally we would like to point out that Theorem \ref{reg2} follows from Theorem 5.3 and Theorem 5.4. 

\medskip
\noindent{\bf Acknowledgements}. The paper was complete while the first author 
was a visiting PhD student of Purdue University. She would like to express her gratitude to
the Department of Mathematics for the hospitality. The second author is partially
supported by NSF grant 1764417.

\end{document}